\documentclass{amsart}
\usepackage[utf8]{inputenc}
\usepackage{amssymb}
\usepackage{hyperref}
\usepackage[final]{showkeys} 

\input xy
\xyoption{all}

\theoremstyle{definition}
\newtheorem{mydef}{Definition}[section]
\newtheorem{lem}[mydef]{Lemma}
\newtheorem{thm}[mydef]{Theorem}

\newtheorem{cor}[mydef]{Corollary}

\newtheorem{defin}[mydef]{Definition}
\newtheorem{example}[mydef]{Example}
\newtheorem{remark}[mydef]{Remark}

\newtheorem{fact}[mydef]{Fact}

\newcommand{\fct}[2]{{}^{#1}#2}

\newcommand{\ba}{\bar{a}}

\newcommand{\bx}{\bar{x}}
\newcommand{\by}{\bar{y}}

\newcommand{\prk}[1]{\operatorname{r}_\ck(#1)}

\newcommand{\ran}{\operatorname{ran}}

\newcommand{\cf}[1]{\text{cf} (#1)}
\newcommand{\seq}[1]{\langle #1 \rangle}
\newcommand{\rest}{\upharpoonright}

\newcommand{\id}{\text{id}}

\newcommand{\leap}[1]{\le_{#1}}

\newcommand{\lea}{\leap{\K}}

\newcommand{\K}{\mathbf{K}}

\newbox\noforkbox \newdimen\forklinewidth
\forklinewidth=0.3pt \setbox0\hbox{$\textstyle\smile$}
\setbox1\hbox to \wd0{\hfil\vrule width \forklinewidth depth-2pt
 height 10pt \hfil}
\wd1=0 cm \setbox\noforkbox\hbox{\lower 2pt\box1\lower
2pt\box0\relax}
\def\unionstick{\mathop{\copy\noforkbox}\limits}

\def\1nf{\unionstick^{(1)}}

\def\2nf{\unionstick^{(2)}}
\def\3nf{\unionstick^{(3)}}

\newcommand{\Hom}{\text{Hom}}

\newcommand{\Ll}{\mathbb{L}}

\newcommand{\tlt}{\triangleleft}

\newcommand{\LS}{\text{LS}}

\newcommand{\PP}{\mathbb{P}}
\newcommand{\bP}{\PP}
\newcommand{\Set}{\operatorname{Set}}

\newcommand{\ck}{\mathcal{K}}
\newcommand{\cL}{\mathcal{L}}
\newcommand{\SCH}{\operatorname{SCH}}
\newcommand{\ESCH}{\operatorname{ESCH}}

\newcommand{\Mod}{\operatorname{Mod}}

\title{Sizes and filtrations in accessible categories}
\date{\today \\
AMS 2010 Subject Classification: Primary 18C35. Secondary: 03C45, 03C48, 03C52, 03C55, 03C75, 03E05.}
\keywords{internal size, presentability rank, existence spectrum, accessibility spectrum, filtrations, singular cardinal hypothesis}

\author[Lieberman]{Michael Lieberman}
\email{lieberman@math.muni.cz}
\urladdr{http://www.math.muni.cz/\textasciitilde lieberman/}
\address{Department of Mathematics and Statistics, Faculty of Science, Masaryk University, Brno, Czech Republic}
\address{Institute of Mathematics, Faculty of Mechanical Engineering, Brno University of Technology, Brno, Czech Republic}

\author[Rosick\'y]{Ji\v r\'i Rosick\'y}
\email{rosicky@math.muni.cz}
\urladdr{http://www.math.muni.cz/\textasciitilde rosicky/}
\address{Department of Mathematics and Statistics, Faculty of Science, Masaryk University, Brno, Czech Republic}
\thanks{The second author is supported by the Grant agency of the Czech republic under the grant 19-00902S}

\author[Vasey]{Sebastien Vasey}
\email{sebv@math.harvard.edu}
\urladdr{http://math.harvard.edu/\textasciitilde sebv/}
\address{Department of Mathematics \\ Harvard University \\ Cambridge, Massachusetts, USA}

\begin{document}

\begin{abstract}
  Accessible categories admit a purely category-theoretic replacement for cardinality: the internal size. Generalizing results and methods from \cite{internal-sizes-jpaa},  we examine set-theoretic problems related to internal sizes and prove several Löwenheim-Skolem theorems for accessible categories. For example, assuming the singular cardinal hypothesis, we show that a large accessible category has an object in all internal sizes of high-enough cofinality. We also prove that accessible categories with directed colimits have filtrations: any object of sufficiently high internal size is (the retract of) a colimit of a chain of strictly smaller objects.
\end{abstract}

\maketitle

\tableofcontents

\section{Introduction}

Recent years have seen a burst of research activity connecting accessible categories with abstract model theory.  Abstract model theory, which has always had the aim of generalizing---in a uniform way---fragments of the rich classification theory of first order logic to encompass the broader nonelementary classes of structures that abound in mathematics proper, is perhaps most closely identified with abstract elementary classes (AECs, \cite{sh88}), but also encompasses metric AECs (mAECs, \cite{maec-hihy}), compact abstract theories (cats, \cite{bymcats}), and a host of other proposed frameworks.  While accessible categories appear in many areas that model theory fears to tread---homotopy theory, for example---they are, fundamentally, generalized categories of models, and the ambition to recover a portion of classification theory in this context has been present since the very beginning, \cite[p.~6]{makkai-pare}.  That these fields are connected has been evident for some time---the first recognition that AECs are special accessible categories came independently in \cite{beke-rosicky} and \cite{lieberman-categ}---but it is only recently that a precise middle-ground has been identified: the $\mu$-AECs of \cite{mu-aec-jpaa}.

While we recall the precise definition of $\mu$-AEC below, we note that they are a natural generalization of AECs in which the ambient language is allowed to be $\mu$-ary, one assumes closure only under $\mu$-directed unions rather than unions of arbitrary chains, and the Löwenheim-Skolem-Tarski property is weakened accordingly.  The motivations for this definition were largely model-theoretic---in a typical AEC, for example, the subclass of $\mu$-saturated models is not an AEC, but does form a $\mu$-AEC---but it turns out, remarkably, that $\mu$-AECs are, up to equivalence, precisely the accessible categories all of whose morphisms are monomorphisms (Fact \ref{maecsacc}). This provides an immediate link between model- and category-theoretic analyses of problems in classification theory, a middle ground in which the tools of each discipline can be brought to bear (and, moreover, this forms the basis of a broader collection of correspondences between $\mu$-AECs with additional properties---universality, admitting intersections---and accessible categories with added structure---locally multipresentable, locally polypresentable \cite{multipres-pams}).

Among other things, this link forces a careful consideration of how one should measure the size of an object: in $\mu$-AECs, we can speak of the cardinality of the underlying set, but we also have a purely category-theoretic notion of \emph{internal size}, which is defined---and more or less well-behaved---in any accessible category (see Definition \ref{presranksize}). This is derived in straightforward fashion from the \emph{presentability rank} of an object $M$, namely the least regular cardinal $\lambda$ (if it exists) such that any morphism sending $M$ into the colimit of a $\lambda$-directed system factors through a component of the system. In most cases, the presentability rank is a successor, and the internal size is then defined to be the predecessor of the presentability rank.

The latter notion generalizes, e.g.\ cardinality in sets (and more generally in AECs), density character in complete metric spaces, cardinality of orthonormal bases in Hilbert spaces, and minimal cardinality of a generator in classes of algebras (Example~\ref{factintsizeex}).  In a sense, one upshot of \cite{internal-sizes-jpaa} is that internal size is the more suitable notion for classification theory, not least because eventual categoricity in power fails miserably, while eventual categoricity in internal size is still very much open.  A related question is that of \emph{LS-accessibility}: in an accessibly category $\ck$, is it the case that there is an object of internal size $\lambda$ for every sufficiently large $\lambda$? Under what ambient set-theoretic assumptions, or concrete category-theoretic assumptions on $\ck$, does this hold? Broadly, approximations to LS-accessibility can be thought of as replacements for the Löwenheim-Skolem theorem in accessible categories. Notice that the analogous statement for cardinality fails miserably: for example, there are no Hilbert spaces whose cardinality has countable cofinality.

One broad aim of the present paper is to relate classical properties of a category (phrased in terms of limits and colimits) to the good behavior of internal sizes in this category. To properly frame these classical properties, we first recall the definition of an accessible category (see Section \ref{prelim-sec} for more details). For $\lambda$ a regular cardinal, a category is \emph{$\lambda$-accessible} if it has $\lambda$-directed colimits, has only a set (up to isomorphism) of $\lambda$-presentable objects (i.e.\ objects with presentability rank at most $\lambda$), and every object can be written as a $\lambda$-directed colimit of $\lambda$-presentable objects. A category is \emph{accessible} if it is $\lambda$-accessible for some $\lambda$. We will say that a category is \emph{large} if it has a proper class of non-isomorphic objects. Thus a category of structures is large exactly when it has objects of arbitrarily large cardinality.

Note that $\lambda$-accessible does not always imply $\lambda'$-accessible for $\lambda' > \lambda$ (see also Fact \ref{raising-acc}). If a given category has this property (i.e.\ it is accessible on a tail of regular cardinal), then we call it \emph{well accessible}. In general, the class of cardinals $\lambda$ such that a given category is $\lambda$-accessible (the \emph{accessibility spectrum}) is a key measure of the complexity of the category. For example, accessible categories with directed colimits \cite[4.1]{beke-rosicky} or $\mu$-AECs with intersections \cite[5.4]{internal-sizes-jpaa} are both known to be well accessible while general $\mu$-AECs  need not be. In the present paper, we attempt to systematically relate the accessibility spectrum to the behavior of internal sizes. For example:

\begin{itemize}
\item We prove that in any well accessible category, high-enough presentability ranks have to be successors (Corollary \ref{well-acc-succ}). This holds more generally of categories where the accessibility spectrum is unbounded below weakly inaccessibles. In particular, we recover the known results that any accessible category with directed colimits \cite[4.2]{beke-rosicky}, any $\mu$-AEC with intersections \cite[5.5(1)]{internal-sizes-jpaa}, and---assuming the singular cardinal hypothesis (SCH)---any accessible category \cite[3.11]{internal-sizes-jpaa}, has high-enough presentability ranks successor.
\item We prove, assuming SCH, that in large accessible categories with all morphisms monos, for all high-enough cardinals $\lambda$, $\lambda^+$-accessibility implies existence of an object of internal size $\lambda$ (Corollary \ref{acc-spectrum-cor}). In this sense, the accessibility spectrum is contained in the existence spectrum. In particular, well accessible categories with all morphisms monos are LS-accessible.
\item Assuming SCH, \emph{any} large accessible category has objects of all internal sizes with high-enough cofinality. In particular, it is weakly LS-accessible (i.e.\ has objects of all high-enough \emph{regular} internal sizes). This is Theorem \ref{esch-no-mono}.
\end{itemize}

Regarding the SCH assumption, we point out that we use a weaker version (``eventual SCH'', see Definition \ref{sch-def}(\ref{sch-def-5})) which follows from the existence of a strongly compact cardinal \cite[20.8]{jechbook}. Thus our conclusions follow from this large cardinal axiom. In reality, we work primarily in ZFC, obtain some local results depending on cardinal arithmetic, and then apply SCH to simplify the statements. Sometimes weaker assumptions than SCH suffice, but we do not yet know whether the conclusions above hold in ZFC itself. Unsurprisingly, dealing with successors of \emph{regular} cardinals is often easier, and can sometimes be done in ZFC.

Another contribution of the present account is the following: throughout the model- and set-theoretic literature, one finds countless constructions that rely on the existence of filtrations, i.e.\ the fact that models can be realized as the union of a continuous increasing chain of models of strictly smaller size.  In a $\lambda$-accessible category, on the other hand, one has that any object can be realized as the colimit of a (more general) $\lambda$-directed system of $\lambda$-presentable objects, but there is no guarantee that one can extract from this system a cofinal chain consisting of objects that are also small.  We here introduce the notion of \emph{well filtrable} accessible category (Definition~\ref{defresolving}), in which the internal size analog of this essential model-theoretic property holds, and show that certain well-behaved classes of accessible categories are well filtrable. We prove general results on existence of filtrations in arbitrary accessible categories (Theorem \ref{filtr-thm}), and deduce that accessible categories with directed colimits \emph{are} well filtrable (or really a slight technical weakening of this, see Corollary \ref{filtr-cor}). This result improves on \cite[Lemma 1]{rosicky-sat-jsl} (which established existence of filtrations only for object of regular internal sizes) and is used in a forthcoming paper on forking independence \cite{more-indep-v2} (a follow-up to \cite{indep-categ-advances}). 

The background required to read this paper is a familiarity with classical set theory (e.g.\ \cite[\S1-8]{jechbook}) and basic category theory (along the lines of \cite{joy-of-cats}); we freely use results and terminology related to accessible categories, e.g.\ \cite{adamek-rosicky,makkai-pare}. The third author's recent introductory paper \cite{bfo-accessible-v1} is a leisurly introduction to most of the required preliminaries. The notion of $\mu$-AEC, whose definition we recall below, first appears in \cite{mu-aec-jpaa}. In a sense, \cite{internal-sizes-jpaa} is also an essential prerequisite for this paper, but we nonetheless try to recall the most essential notions here to make the paper as self-contained as possible. We will also give the proof of a known result if it can be derived in a particularly straightforward way from results given here. The first few sections contain some basic model-theoretic results on $\mu$-AECs that nevertheless have not appeared exactly in this form before. In particular, we reprove the presentation theorem for $\mu$-AECs, fixing a mistake in \cite[\S3]{mu-aec-jpaa} reported to us by Marcos Mazari-Armida. We would like to thank him again for his thorough reading of our earlier work.  We also thank Mike Shulman and a referee for helpful suggestions, questions and comments.

\section{Preliminaries}\label{prelim-sec}

We start by recalling the necessary set-theoretic notation.

\begin{defin}\label{sch-def}
  Let $\lambda$ and $\mu$ be infinite cardinals with $\mu$ regular.
  \begin{enumerate}
  \item For $A$ a set, we write $[A]^{<\lambda}$ for the set of all subsets of $A$ of cardinality strictly less than $\lambda$, and similarly define $[A]^\lambda$ (we will sometimes think of it as being partially ordered by containment). For $B$ a set, we write $\fct{B}{A}$ for the set of all functions from $B$ to $A$, and let $\fct{<\lambda}{A} := \bigcup_{\alpha < \lambda} \fct{\alpha}{A}$.
  \item A partially ordered set (poset for short) $I$ is \emph{$\mu$-directed} (for $\mu$ a regular cardinal) if any subset of $I$ of cardinality strictly less than $\mu$ has an upper bound. When $\mu = \aleph_0$, we omit it.
  \item We write $\lambda^-$ for the predecessor of the cardinal $\lambda$, defined as follows:
    
		$$\lambda^-=\left\{\begin{array}{ll} \theta & \mbox{if }\lambda=\theta^+\\
		  \lambda & \mbox{if }\lambda\mbox{ is limit}\end{array}\right.$$
  \item We say that\footnote{Lurie \cite[A.2.6.3]{htt-lurie} writes $\mu \ll \lambda$ to mean that $\mu < \lambda$ are both regular and $\lambda$ is $\mu$-closed.} $\lambda$ is \emph{$\mu$-closed} if $\theta^{<\mu} < \lambda$ for all $\theta < \lambda$. 
  \item\label{tlt-def} For $\mu < \lambda$ regular, we write $\mu \tlt \lambda$ if $\cf{[\theta]^{<\mu}]} < \lambda$ for all $\theta < \lambda$ (see \cite[2.12]{adamek-rosicky}). We extend this notion to singular cardinals $\lambda$ as follows\footnote{This was suggested by Mike Shulman.}: when $\lambda$ is singular, we write $\mu \tlt \lambda$ if for all $\lambda_0 < \lambda$, there exists a regular $\lambda_0' \in (\lambda_0, \lambda]$ so that $\mu \tlt \lambda_0'$.
  \item When we write a statement like ``for all high-enough $\theta$, ...'', we mean ``there exists a cardinal $\theta_0$ such that for all $\theta \ge \theta_0$, ...''.
  \item\label{sch-def-5} We say that \emph{$\SCH$ holds at $\lambda$} if $\lambda^{\cf{\lambda}} = 2^{\cf{\lambda}} + \lambda^+$ ($\SCH$ stands for the \emph{singular cardinal hypothesis} --- note that the equation is always true for regular $\lambda$). We say that \emph{$\SCH$ holds above $\lambda$} if $\SCH$ holds at $\theta$ for all cardinals $\theta \ge \lambda$. The \emph{eventual singular cardinal hypothesis} ($\ESCH$) is the statement ``$\SCH$ holds at all high-enough $\theta$'', or more precisely ``there exists $\theta_0$ such that $\SCH$ holds above $\theta_0$''.
  \end{enumerate}
\end{defin}
\begin{remark}\label{tlt-rmk}
  Let $\mu$ be a regular cardinal. If $\lambda > \mu$ is $\mu$-closed, then $\mu \tlt \lambda$  (because $\cf{[\theta]^{<\mu}} \le \theta^{<\mu}$). Conversely, if $\lambda > 2^{<\mu}$ and $\mu \tlt \lambda$, then $\lambda$ is $\mu$-closed (see \cite[2.5]{internal-sizes-jpaa}). Note also that if $\mu \tlt \lambda$ and $\lambda$ is weakly inaccessible, then an easy closure argument shows that there exists unboundedly-many successor cardinals $\lambda_0 < \lambda$ such that $\mu \tlt \lambda_0$. In fact, in the present paper, we can sometimes replace an assumption of the form ``$\lambda$ is $\mu$-closed'' with the slightly more precise ``$\mu \tlt \lambda$'', however for simplicity we will rarely do so.
\end{remark}

It is a result of Solovay (see \cite[20.8]{jechbook}) that $\SCH$ holds above a strongly compact cardinal. Thus $\ESCH$ follows from this large cardinal axiom. We will assume $\ESCH$ in several results of the present paper.

The facts below are well-known to set theorists. See \cite[\S5]{jechbook}.

\begin{fact}\label{sch-basic} \
  \begin{enumerate}
  \item If $\lambda$ is a $\mu$-closed cardinal ($\mu$ regular), then $\lambda = \lambda^{<\mu}$ if and only if $\lambda$ has cofinality at least $\mu$.
  \item If $\SCH$ holds above an infinite cardinal $\theta$, then for every cardinal $\lambda$ and every regular cardinal $\mu$, $\lambda^{<\mu} \le \lambda^+ + \sup_{\theta_0 < \theta} \theta_0^{<\mu}$. In particular, every cardinal strictly greater than $\sup_{\theta_0 < \theta} \theta_0^{<\mu}$ which is not the successor of a cardinal of cofinality strictly less than $\mu$ is $\mu$-closed. 
  \end{enumerate}
\end{fact}

The following easy result will be used in the proof of Lemma \ref{bounded-pres-nonex}

\begin{lem}\label{cofin-dir}
  Let $\mu$ be a regular cardinal and let $I$ be a $\mu$-directed poset. Let $\alpha < \mu$ and let $\seq{I_i : i < \alpha}$ be a sequence of subsets of $I$. If $I = \bigcup_{i < \alpha} I_i$, then there exists $i < \alpha$ such that $I_i$ is cofinal in $I$.
\end{lem}
\begin{proof}
  Suppose not. Then for each $i < \alpha$ there exists $p_i \in I$ such that no element of $I_i$ bounds $p_i$. Since $I$ is $\mu$-directed, there exists $p \in I$ such that $p_i \le p$ for all $i < \alpha$. There is $i \in I$ such that $p \in I_i$, contradicting the choice of $p_i$.
\end{proof}

The ideas at the heart of the category-theory-enriched form of classification theory at work here, in \cite{internal-sizes-jpaa}, and in \cite{indep-categ-advances}, are the notions of \emph{presentability rank} and \emph{internal size}.

\begin{defin}\label{presranksize}
	Let $\lambda$ and $\mu$ be infinite cardinals, $\mu$ regular, and let $\ck$ be a category.
	\begin{enumerate}
		\item We say that a diagram $D:I\to\ck$ is \emph{$\mu$-directed} if $I$ is a $\mu$-directed poset. A \emph{$\mu$-directed colimit} is just the colimit of a $\mu$-directed diagram.
		\item We say that an object $M$ in $\ck$ is \emph{$\mu$-presentable} if the hom-functor $$\Hom_\ck(M,-):\ck\to\Set$$
		preserves $\mu$-directed colimits.  Equivalently, $M$ is $\mu$-presentable if every morphism $f:M\to N$, with $N$ a $\mu$-directed colimit with cocone $\langle N_i\stackrel{d_i}{\to}N\,|\,i\in I\rangle$, the map $f$ factors essentially uniquely through one of the $d_i$'s.
		\item We say that $M$ is \emph{$(<\lambda)$-presentable} if it is $\theta$-presentable for some regular $\theta<\lambda + \aleph_1$.
		\item The \emph{presentability rank} of an object $M$ in $\ck$, denoted $\prk{M}$, is the smallest $\mu$ such that $M$ is $\mu$-presentable. We sometimes drop $\ck$ from the notation if it is clear from context.
		\item The \emph{internal size} of $M$ in $\ck$ is defined to be $|M|_\ck=\prk{M}^-$. Again, we may drop $\ck$ from the notation if it is clear from context. 
	\end{enumerate}
\end{defin}

\begin{example}\label{factintsizeex} We recall that internal size corresponds to the natural notion of size in familiar categories:
\begin{itemize}
	\item In the category of sets, the internal size of any infinite set is precisely its cardinality.  In an AEC, too, the internal size of any sufficiently big model will be its cardinality (see \cite[4.3]{lieberman-categ} or Fact \ref{rank-bounds} here).
	\item In the category of complete metric spaces and contractions, the internal size of any infinite space is its density character (the minimal cardinality of a dense subset).  This is true, as well, for sufficiently big models in a general metric AEC, \cite[3.1]{lrcaec-jsl}.
	\item In the category of Hilbert spaces and linear contractions, the internal size of any infinite dimensional space is the cardinality of its orthonormal basis.
        \item In the category of free algebras with exactly one $\omega$-ary function, the internal size is the minimal cardinality of a generator. In fact, a similar characterization holds in any $\mu$-AEC with a notion of generation (i.e.\ with intersections), see \cite[5.7]{internal-sizes-jpaa}.
\end{itemize}	
\end{example}

As the above examples indicate, the relationship between internal size and cardinality can be very delicate---particularly in a context as general as $\mu$-AECs or, equivalently, accessible categories with monomorphisms (henceforth \emph{monos})---and seems to become tractable only under mild set- or category-theoretic assumptions.  This is the substance of \cite{internal-sizes-jpaa}, results of which we refine in the present paper.

We recall from \cite[3.1, 3.3]{internal-sizes-jpaa} an essential piece of terminology:

\begin{defin}\label{systems-def}
	Let $\lambda$ and $\mu$ be infinite cardinals, $\mu$ regular.
	\begin{enumerate}
	\item A \emph{$(\mu,<\lambda)$-system} in a category $\ck$ is a $\mu$-directed diagram consisting of $(<\lambda)$-presentable objects. A \emph{$(\mu, \lambda)$-system} (for $\lambda$ regular) is a $(\mu, <\lambda^+)$-system.
	\item We say that a $(\mu,<\lambda)$-system with colimit $M$ is \emph{proper} if the identity map on $M$ does not factor through any object in the system.
	\end{enumerate}
\end{defin}

The following two results on the relationship between presentability and directedness of systems are basic.

\begin{fact}\label{sys-basic-facts}
  Let $\lambda$, $\mu$, and $\theta$ be infinite cardinals, $\mu$ regular. Let $\ck$ be a category with $\mu$-directed colimits.

  \begin{enumerate}
  \item\label{sys-basic-facts-1} \cite[3.5]{internal-sizes-jpaa} For $\lambda$ a regular cardinal, the colimit of a $(\mu, \lambda)$-system with $\theta$ objects is always $(\theta^+ + \lambda)$-presentable. In fact (for $\lambda$ not necessarily regular), if $\cf{\lambda} > \theta$ and $\lambda$ is not the successor of a singular cardinal, the colimit of a $(\mu, <\lambda)$-system with $\theta$ objects is always $(<(\theta^{++} + \lambda))$-presentable.
  \item\label{sys-basic-facts-2} \cite[3.4]{internal-sizes-jpaa} Let $M$ be the colimit of a $(\mu, <\lambda)$-system.

    \begin{enumerate}
    \item If $M$ is $\mu$-presentable, then the system is not proper.
    \item If the system is not proper, then $M$ is $(<\lambda)$-presentable.
    \end{enumerate}
  \end{enumerate}
\end{fact}

We hereby obtain a criterion for the existence of objects whose presentability rank is the successor of a regular cardinal (this is already implicit in the proof of \cite[3.12]{internal-sizes-jpaa}):

\begin{cor}\label{succ-reg-cor}
  Let $\mu$ be a regular cardinal. In a category with $\mu$-directed colimits, the colimit of a proper $(\mu, \mu^+)$-system containing at most $\mu$ objects has presentability rank $\mu^+$.
\end{cor}
\begin{proof}
  Let $M$ be this colimit. By Fact \ref{sys-basic-facts}(\ref{sys-basic-facts-1}), $M$ is $\mu^+$-presentable. By Fact \ref{sys-basic-facts}(\ref{sys-basic-facts-2}), $M$ is not $\mu$-presentable.
\end{proof}

The notion of a $(\mu, <\lambda)$-system also allows a more parameterized and compact rephrasing of the definition of an accessible category:

\begin{defin}\label{acccat-defs} Let $\ck$ be a category and $\lambda$ and $\mu$ be infinite cardinals, $\mu$ regular.\begin{enumerate}
	\item \cite[3.6]{internal-sizes-jpaa} We say that $\ck$ is $(\mu,<\lambda)$-accessible if it has the following properties:
		\begin{enumerate} 
		\item $\ck$ has $\mu$-directed colimits.
		\item $\ck$ contains a set of $(<\lambda)$-presentable objects, up to isomorphism.
		\item Any object in $\ck$ is the colimit of a $(\mu,<\lambda)$-system.
		\end{enumerate}
              \item We say that $\ck$ is  \emph{$(\mu, \lambda)$-accessible} if $\lambda$ is regular and $\ck$ is $(\mu, <\lambda^+)$-accessible. We say that $\ck$ is \emph{$\mu$-accessible} if it is $(\mu, \mu)$-accessible (this corresponds to the usual definition from, e.g.\ \cite{adamek-rosicky, makkai-pare}). We sometimes say \emph{finitely accessible} instead of $\aleph_0$-accessible.
              \item\label{well-acc-def} \cite[2.1]{beke-rosicky} We say that $\ck$ is \emph{well $\mu$-accessible} if it is $\theta$-accessible for each regular cardinal $\theta \ge \mu$. We say that $\ck$ is \emph{well accessible} if it is well $\mu_0$-accessible for some regular cardinal $\mu_0$.

  \end{enumerate}
\end{defin}

We will use the following result, allowing us to change the index of accessibility of a category (see \cite[2.3.10]{makkai-pare} or \cite[3.8]{internal-sizes-jpaa}):

\begin{fact}\label{raising-acc}
  Let $\ck$ be a $(\mu, <\lambda)$-accessible category. If $\theta$ is regular and $\mu \tlt \theta$, then $\ck$ is $(\theta, <(\lambda + \theta^+))$-accessible. Moreover, every object of $\ck$ can be written as a $\theta$-directed diagram, where each object is a $\mu$-directed colimit of strictly fewer than $\theta$-many $(<\lambda)$-presentables.
\end{fact}

The following two definitions describe good behavior of the existence spectrum of an accessible category (LS-accessibility appears in \cite[2.4]{beke-rosicky}, weak LS-accessibility is introduced in \cite[A.1]{internal-sizes-jpaa}):

\begin{defin}\label{ls-acc-def}
  An accessible category $\ck$ is \emph{LS-accessible} if it has objects of all high-enough successor presentability ranks. We say that $\ck$ is \emph{weakly LS-accessible} if it has objects of all high-enough presentability ranks that are successors of \emph{regular} cardinals. 
\end{defin}
  
Similar to an accessible category, a $\mu$-AEC is an abstract class of structures in which any model can be obtained by sufficiently highly directed colimits of small objects:

\begin{defin}[{\cite[\S2]{mu-aec-jpaa}}] Let $\mu$ be a regular cardinal. \begin{enumerate}
  \item A \emph{($\mu$-ary) abstract class} is a pair $\K = (K, \lea)$ such that $K$ is a class of structures in a fixed $\mu$-ary vocabulary $\tau = \tau (\K)$, and $\lea$ is a partial order on $K$ that respects isomorphisms and extends the $\tau$-substructure relation. For $M \in \K$ we write $U M$ for the universe of $M$.
  \item An abstract class $\K$ is a \emph{$\mu$-abstract elementary class} (or \emph{$\mu$-AEC} for short) if it satisfies the following three axioms:

  \begin{enumerate}
  \item Coherence: for any $M_0, M_1, M_2 \in \K$, if $M_0 \subseteq M_1 \lea M_2$ and $M_0 \lea M_2$, then $M_0 \lea M_1$.
  \item Chain axioms: if $\seq{M_i : i \in I}$ is a $\mu$-directed system in $\K$, then:
    \begin{enumerate}
    \item $M := \bigcup_{i \in I} M_i$ is in $\K$.
    \item $M_i \lea M$ for all $i \in I$.
    \item If $M_i \lea N$ for all $i \in I$, then $M \lea N$.
    \end{enumerate}
  \item Löwenheim-Skolem-Tarski (LST) axiom: there exists a cardinal $\lambda = \lambda^{<\mu} \ge |\tau (\K)| + \mu$ such that for any $M \in \K$ and any $A \subseteq U M$, there exists $M_0 \in \K$ with $M_0 \lea M$, $A \subseteq U M_0$, and $|U M_0| \le |A|^{<\mu} + \lambda$. We write $\LS (\K)$ for the least such $\lambda$.
  \end{enumerate}

  When $\mu = \aleph_0$, we omit it and call $\K$ an \emph{abstract elementary class} (AEC for short).
  \end{enumerate}
\end{defin}

One can see any $\mu$-AEC $\K$ (or, indeed, any $\mu$-ary abstract class) as a category in a natural way: a morphism between models $M$ and $N$ in $\K$ is a map $f:M\to N$ which induces an isomorphism from $M$ onto $f[M]$, and such that $f[M] \lea N$.  We abuse notation slightly: we will still use boldface when referring to this category, i.e.\ we denote it by $\K$ and not $\ck$, to emphasize the concreteness of the category.

As mentioned in the introduction, these classes are an ideal locus of interaction between abstract model theory and accessible categories:

\begin{fact}[{\cite[\S4]{mu-aec-jpaa}}]\label{maecsacc}
  Any $\mu$-AEC $\K$ with $\LS (\K)=\lambda$ is a $\lambda^+$-accessible category with $\mu$-directed colimits and all morphisms monos.  Conversely, any $\mu$-accessible category $\ck$ with all morphisms monos is equivalent (as a category) to a $\mu$-AEC $\K$ with $\LS (\K)\le\max(\mu,\nu)^{<\mu}$, where $\nu$ is the size of the full subcategory of $\ck$ on the set of (representatives) of $\mu$-presentable objects.
\end{fact}

We finish with two facts on internal sizes in $\mu$-AECs. The first describes the relationship between presentability and cardinality:

\begin{fact}\label{rank-bounds}
  Let $\K$ be a $\mu$-AEC and let $M \in \K$. Then:

  \begin{enumerate}
  \item\cite[4.5]{internal-sizes-jpaa} $r_{\K} (M) \le |U M|^+ + \mu$.
  \item\label{rank-bounds-2} If $\lambda > \LS (\K)$ is a $\mu$-closed cardinal such that $M$ is $(<\lambda^+)$-presentable, then $|U M| < \lambda$.
  \end{enumerate}
\end{fact}
\begin{proof}[Proof of (\ref{rank-bounds-2})]
  If $\lambda$ is singular, we can replace $\lambda$ by a regular $\lambda_0 \in [\LS (\K)^+, \lambda)$ such that $M$ is $\lambda_0$-presentable and $\lambda_0$ is $\mu$-closed, so without loss of generality, $\lambda$ is regular. Now apply \cite[4.8, 4.9]{internal-sizes-jpaa}.
\end{proof}

The second gives a sufficient condition for existence of objects of presentability rank the successor of a regular cardinal. The proof is very short given what has already been said, so we give it.

\begin{fact}[{\cite[3.12]{internal-sizes-jpaa}}]\label{succ-reg-existence}
  Let $\lambda$ and $\mu$ be infinite cardinals with $\mu$ regular. If $\ck$ is a $(\mu, <\lambda)$-accessible category with all morphisms monos and $\ck$ has an object that is not $(<\lambda)$-presentable, then $\ck$ has a proper $(\mu, <\lambda)$-system with $\mu$ objects. In particular, if in addition $\lambda \le \mu^{++}$, then $\ck$ has an object of presentability rank $\mu^+$.
\end{fact}
\begin{proof}
  Let $N$ be an object of $\ck$ that is not $(<\lambda)$-presentable. Since $\ck$ is $(\mu, <\lambda)$-accessible, $N$ is the colimit of a $(\mu, <\lambda)$-system $\seq{N_i : i \in I}$. Since $N$ is not $(<\lambda)$-presentable, the system is proper (Fact \ref{sys-basic-facts}(\ref{sys-basic-facts-2})). Since $I$ is $\mu$-directed and all morphisms of $\ck$ are monos, we can pick a strictly increasing chain $\seq{i_k : k < \mu}$ inside $I$ such that $\seq{M_{i_k} : k < \mu}$ is still proper. This then gives the desired proper $(\mu, <\lambda)$-system with $\mu$ objects. The ``in particular'' part follows from Corollary \ref{succ-reg-cor}.
\end{proof}

\section{Directed systems and cofinal posets}

As observed in \cite[4.1]{mu-aec-jpaa}, if $\K$ is a $\mu$-AEC then any $M \in \K$ is the $\LS (\K)^+$-directed union of all its $\K$-substructures of cardinality at most $\LS (\K)$. In an AEC (i.e.\ when $\mu = \aleph_0$), it is well known that furthermore one can write $M = \bigcup_{s \in [U M]^{<\mu}} M_s$, where the $M_s$'s form a $\mu$-directed system of objects of cardinality at most $\LS (\K)$, and $s \subseteq U M_s$ for all $s$. In the proof of \cite[3.2]{mu-aec-jpaa}, it was asserted without proof that the corresponding statement was also true when $\mu > \aleph_0$. This was a key ingredient of the proof of the presentation theorem there. It was pointed out to us (by Marcos Mazari-Armida) that the proof for AECs does \emph{not} generalize to $\mu$-AECs: since we cannot take unions, there are problems at limit steps. Thus we in fact do not know whether the statement is still true for $\mu$-AECs. In this section, we prove a weakening and in the next two sections reprove the presentation theorem and related axiomatizability results.

We will more generally develop the theory of subposets of posets that are cofinal in a generalized sense:

\begin{defin}
  A \emph{partially ordered set} (or poset) is a binary relation $(\bP, \le)$ which is transitive, reflexive, and antisymmetric. We may not always explicitly mention $\le$. A \emph{subposet} of a partially ordered set $\bP$ is a poset $(\bP^\ast, \le^\ast)$ with $\bP^\ast \subseteq \bP$ and $x \le^\ast y$ implying $x \le y$. 
\end{defin}

\begin{defin}\label{cofinal-def}
  For $\theta$ a cardinal, a subposet $\bP^\ast$ of a poset $\bP$ is \emph{$\theta$-cofinal} if for any $p \in \bP$ and any sequence $\seq{p_i : i < \theta}$ of elements of $\bP^\ast$ with $p_i \le p$ for all $i < \theta$, there exists $q \in \bP^\ast$ such that $p \le q$ and $p_i \le^\ast q$ for all $i < \theta$. We say that $\bP^\ast$ is \emph{$(<\theta)$-cofinal} if it is $\theta_0$-cofinal for all $\theta_0 < \theta$.
\end{defin}

Note that $\bP^\ast$ is $0$-cofinal if and only if it is cofinal in $\bP$ as a set. Being $1$-cofinal means that if $p_0 \le p$ with $p \in \bP$ and $p_0 \in \bP^\ast$, there exists $q \in \bP^\ast$ so that $p \le q$ \emph{and} $p_0 \le^\ast q$. If $\bP^\ast$ is $1$-cofinal, an induction shows that it is automatically $n$-cofinal for all $n < \omega$ (a similar argument appears in \cite[3.9]{multidim-v2}). More generally, we can get that it is $\theta$-cofinal assuming chain bounds (a poset has \emph{$\alpha$-chain bounds} if any chain of length $\alpha$ has an upper bound; define $(<\alpha)$-chain bounds, $(<\le \alpha)$-chain bounds, in the expected way):

\begin{lem}\label{chain-cofinal}
  Let $\theta$ be a cardinal and let $\bP^\ast$ be a subposet of a poset $\bP$. If $\bP^\ast$ is $1$-cofinal in $\bP$ and $\bP^\ast$ has $(\le \theta)$-chain bounds, then $\bP^\ast$ is $\theta$-cofinal in $\bP$.
\end{lem}
\begin{proof}
  Let $p \in \bP$ and $\seq{p_i : i < \theta}$ be below $p$ and in $\bP^\ast$. We proceed by induction on $\theta$. If $\theta \le 1$, the result is true by assumption. Assume now that $\theta > 1$ and the result holds for all $\theta_0 < \theta$. We build an increasing chain $\seq{q_i : i < \theta}$ in $\bP^\ast$ such that $p \le q_0$ and for all $j < i < \theta$, $p_j \le^\ast q_i$. This is possible by successive use of the induction hypothesis and $1$-cofinality (at limits, use the chain bound assumption). Now take $q \in \bP^\ast$ $\le^\ast$-above each $q_i$.
\end{proof}

Definition \ref{cofinal-def} helps us understand the relationship between $\theta$-directedness of $\bP^\ast$ and $\bP$. The proof is straightforward, so we omit it.

\begin{lem}\label{dir-lem}
  Let $\bP^\ast$ be a $0$-cofinal subposet of a poset $\bP$ and let $\theta$ be an infinite cardinal. 
  \begin{enumerate}
  \item If $\bP^\ast$ is $\theta$-directed, then $\bP$ is $\theta$-directed.
  \item If $\bP^\ast$ is $(<\theta)$-cofinal and $\bP$ is $\theta$-directed, then $\bP^\ast$ is $\theta$-directed.
  \end{enumerate}
\end{lem}

The next result extracts, from a diagram in $\bP$, a certain cofinal extension of that diagram in $\bP^\ast$. It will be applied to $\mu$-AECs.

\begin{thm}\label{dir-thm}
  Let $\theta$ be an infinite cardinal, let $\bP$ be a $\theta$-directed poset, and let $\bP^\ast$ be a $(<\theta)$-cofinal subposet of $\bP$. Let $(I, \le)$ be a partial order and let $\seq{p_i : i \in I}$ be a diagram in $\bP$ (that is, if $i \le j$ are in $I$, then $p_i \le p_j$). If for all $i \in I$, $|\{j \in I \mid j < i\}| < \theta$, then there exists $I_0 \subseteq I$ cofinal and a diagram $\seq{q_i : i \in I_0}$ in $\bP^\ast$ such that for all $i \in I_0$, $p_i \le q_i$.
\end{thm}
\begin{proof}
  We define a new poset $\mathcal{F}$. Its objects are diagrams $\seq{q_i : i \in I_0}$ in $\bP^\ast$, with $I_0 \subseteq I$ (not necessarily cofinal) and with $p_i \le q_i$ for all $i \in I_0$. Order $\mathcal{F}$ by extension. Now $\mathcal{F}$ is not empty (the empty diagram is in $\mathcal{F}$) and every chain in $\mathcal{F}$ has an upper bound (its union). By Zorn's lemma, there is a maximal element $\seq{q_i : i \in I_0}$ in $\mathcal{F}$. We show that $I_0$ is cofinal in $I$. Suppose not and let $i^\ast \in I$ be such that $i^\ast \not \le j$ for any $j \in I_0$.

  Since $\bP^\ast$ is $0$-cofinal in $\bP$, there exists $p_{i^\ast}'$ in $\bP^\ast$ such that $p_{i^\ast} \le p_{i^\ast}'$. Consider the set $Q := \{p_{i^\ast}'\} \cup \{q_i \mid i \in I_0, i < i^\ast\}$. Note that $|Q| < \theta$ and by Lemma \ref{dir-lem} $\bP^\ast$ is $\theta$-directed, so pick $q_{i^\ast} \in \bP^\ast$ an upper bound for $Q$. 

  Consider $\seq{q_i : i \in I_0 \cup \{i^\ast\}}$. We show this is an element of $\mathcal{F}$, contradicting the maximality of $\seq{q_i : i \in I_0}$. First, it is clear from the definition that $p_i \le q_i$ for all $i \in I_0 \cup \{i^\ast\}$. It remains to see that $\seq{q_i : i \in I_0 \cup \{i^\ast\}}$ is a diagram in $\bP^\ast$. So let $i \le j$ be two elements of $I_0 \cup \{i^\ast\}$. If $i, j \in I_0$, then we are done by the assumption that $\seq{q_i : i \in I_0}$ is a diagram. Thus at least one of $i$ or $j$ is equal to $i^\ast$. If $i = j = i^\ast$, then since $q_{i^\ast} \in \bP^\ast$ we are also done. If $i \in I_0$ and $j = i^\ast$, we have made sure in the construction that $q_i \le^\ast q_j$. Finally, we cannot have that $i = i^\ast$ and $j \in I_0$ by the choice of $i^\ast$ (a witness to the non-cofinality of $I_0$).
\end{proof}

As an application, we can study sufficiently closed objects $M$ in an abstract class. While we may not be able to resolve such an object with a system indexed by $[U M]^{<\theta}$, we can at least get a system indexed by a \emph{cofinal} subset of $[U M]^{<\theta}$:

\begin{defin}
  Let $\K$ be an abstract class and let $\theta$ be an infinite cardinal. An object $M \in \K$ is \emph{$\theta$-closed} if for any $A \in [U M]^{<\theta}$, there exists $M_0 \in \K$ with $M_0 \lea M$, $A \subseteq U M_0$, and $|U M_0| < \theta$.
\end{defin}

\begin{cor}\label{resolv-lem}
  Let $\K$ be an abstract class, let $\theta$ be an infinite cardinal, and let $M \in \K$ be $\theta$-closed. Let $I \subseteq [U M]^{<\theta}$ be such that for any $s \in I$, $|\mathcal{P} (s) \cap I| < \cf{\theta}$. Then there exists $I_0 \subseteq I$ and $\seq{M_s : s \in I_0}$ such that $I_0$ is cofinal in $I$ and for any $s, t \in I_0$:

  \begin{enumerate}
  \item\label{resolv-lem-1} $M_s \le M$.
  \item $|U M_s| < \theta$.
  \item\label{resolv-lem-3} $s \subseteq U M_s$.
  \item\label{resolv-lem-4} $s \subseteq t$ implies $U M_s \subseteq  U M_t$.
  \end{enumerate}
\end{cor}
\begin{proof}
  Let $\bP$ be the partially ordered set $[U M]^{<\theta}$, ordered by subset inclusion. Note that $\bP$ is $\theta$-directed. Let $\bP^\ast$ be the partially ordered set $\{U M_0 \mid M_0 \in \K, M_0 \lea M, |U M_0| < \theta\}$, also ordered by subset inclusion. It is easy to check that $\bP^\ast$ is a subposet of $\bP$. It is also straightforward to see that $\bP^\ast$ is $1$-cofinal in $\bP$ (using that $M$ is $\theta$-closed). Similarly, $\bP^\ast$ has $(<\cf{\theta})$-chain bounds. By Lemma \ref{chain-cofinal}, $\bP^\ast$ is $(<\cf{\theta})$-cofinal in $\bP$. Apply Theorem \ref{dir-thm}, with $\theta$ there being $\cf{\theta}$ here and $p_s = s$ for each $s \in I$.
\end{proof}

As an application of Corollary \ref{resolv-lem}, we study what happens if we weaken the Löwenheim-Skolem-Tarski (LST) axiom of $\mu$-AECs to the ``weak LST axiom'': there exists a cardinal $\lambda \ge |\tau (\K)| + \mu$ such that $\lambda = \lambda^{<\mu}$ and every object of $\K$ is $\lambda^+$-closed (i.e.\ for all $M \in \K$ and all $A \subseteq U M$ \emph{of cardinality at most $\lambda$}, there exists $M_0 \in \K$ with $M_0 \lea M$, $|U M| \le \lambda$, and $A \subseteq U M$). It was shown in \cite[4.6]{mu-aec-jpaa} that such a weakening still implies the original LST axiom, but the proof did not give that the minimal $\lambda$ satisfying the weak LST axiom should be the Löwenheim-Skolem-Tarski number. We prove this now.

\begin{cor}
  Let $\mu$ be a regular cardinal, let $\K$ be an abstract class satisfying the coherence and chain axioms of $\mu$-AECs, and let $\lambda$ be an infinite cardinal. If $\lambda = \lambda^{<\mu}$ and any element of $\K$ is $\lambda^+$-closed, then $\K$ is a $\mu$-AEC with $\LS (\K) \le \lambda$.  
\end{cor}
\begin{proof}
  Let $M \in \K$ and let $A \subseteq U M$. Apply Corollary \ref{resolv-lem} with $I := [A]^{<\mu}$ and $\theta := \lambda^+$ (note that $2^{<\mu} \le \lambda^{<\mu} = \lambda < \theta$, so the cardinal arithmetic condition there is satisfied). Let $\seq{M_s : s \in I_0}$ be as given there. This is a $\mu$-directed system by coherence, so let $N$ be its union. By construction, $N \le M$ and $A \subseteq U N$. Moreover, $|U N| \le |I_0| \cdot \lambda \le |A|^{<\mu} + \lambda$, as needed. 
\end{proof}

\section{Presentation theorem and axiomatizability}

We reprove here the presentation theorem for $\mu$-AECs (and more generally for accessible categories with $\mu$-directed colimits and all morphisms monos), in the form outlined and motivated in \cite[\S6]{multipres-pams} (there, additional assumptions on the existence of certain directed colimits had to be inserted to make the proof work). The idea is simple: any $\mu$-accessible category is equivalent to the category of models of an $\Ll_{\infty, \mu}$-sentence, and we can Skolemize such a sentence to obtain the desired functor. We first state the three facts we will use. Recall that $\Mod (\phi)$ denotes the category of models of the sentence $\phi$, with morphisms all homomorphisms (i.e.\ maps preserving functions and relations). See also \cite[\S4]{multipres-pams} for a summary of what is known on axiomatizability of accessible categories.

\begin{fact}[{\cite[3.2.3, 3.3.5, 4.3.2]{makkai-pare}}]\label{acc-mod-phi}
  Any $\mu$-accessible category is equivalent to $\Mod (\phi)$, for $\phi$ an $\Ll_{\infty, \mu}$-formula.
\end{fact}

Recall \cite[2.1]{multipres-pams} that (for a regular cardinal $\mu$) a \emph{$\mu$-universal class} is a class of structures in a $\mu$-ary vocabulary that is closed under isomorphisms, substructure, and $\mu$-directed unions. 

\begin{fact}[Skolemization]\label{skolem}
  Let $\mu \le \lambda$. If $\phi$ is an $\Ll_{\lambda^+, \mu}$-sentence in the vocabulary $\tau$, there exists an expansion $\tau^+$ of $\tau$ with function symbols and a $\mu$-universal class $K^+$ with vocabulary $\tau^+$ such that:

  \begin{enumerate}
  \item $|\tau^+| \le \lambda$.
  \item The reduct map is a faithful functor from $K^+$ into $\Mod (\phi)$ that is surjective on objects and preserves $\mu$-directed colimits.
  \end{enumerate}
\end{fact}
\begin{proof}[Proof sketch]
  Let $\Phi$ be a fragment (i.e.\ set of formulas in $\Ll_{\lambda^+, \mu}$ closed under subformulas) which contains $\phi$ and has cardinality at most $\lambda$. Add a Skolem function for each formula in $\Phi$, forming a vocabulary $\tau^+$, and let $K^+$ be the set of all $\tau^+$-structures whose reduct is a model of $\phi$ and where the Skolem functions perform as expected. 
\end{proof}

The result below was stated for $\mu = \aleph_0$ in \cite[2.5]{ct-accessible-jsl}, but the proof easily generalizes.

\begin{fact}[{\cite[2.5]{ct-accessible-jsl}}]\label{mu-acc}
  If $\ck$ is an accessible category with $\mu$-directed colimits and all morphisms monos, there exists a $\mu$-accessible category $\cL$ and a faithful essentially surjective functor $F: \cL \to \ck$ preserving $\mu$-directed colimits.  In fact, if $\ck$ is $\lambda$-accessible, $\cL$ is the free completion under $\mu$-directed colimits of the full subcategory of $\ck$ induced by its $\lambda$-presentable objects.
\end{fact}

\begin{thm}[The presentation theorem for $\mu$-AECs]\label{pres-thm}
  If $\ck$ is an accessible category with $\mu$-directed colimits and all morphisms monos, then there exists a $\mu$-universal class $\cL$ and an essentially surjective faithful functor $F: \cL \to \ck$ preserving $\mu$-directed colimits.
\end{thm}
\begin{proof}
  By Fact \ref{mu-acc}, there exists a $\mu$-accessible category $\ck^1$ and a faithful essentially surjective functor $F^1 : \ck^1 \to \ck$. By Fact \ref{acc-mod-phi}, $\ck^1$ is equivalent to $\Mod (\phi)$, for some $\Ll_{\infty, \mu}$-sentence $\phi$. Since all morphisms are monos, we may assume that non-equality is part of the vocabulary of $\phi$. By Fact \ref{skolem}, we can find a $\mu$-universal class $K^+$ so that the reduct map $F^0 : K^+ \to \Mod (\phi)$ is an essentially surjective faithful functor preserving $\mu$-directed colimits. Set $\cL := K^+$, $F := F^1 \circ F^0$.
\end{proof}

Note that if we apply Theorem \ref{pres-thm} to a $\mu$-AEC $\K$, the functor is \emph{not} directly given by a reduct from some expansion of $\K$ (we first have to pass through several equivalence of categories). Thus Theorem \ref{pres-thm} does not immediately prove (as in \cite{tamelc-jsl} for AECs) that $\mu$-AECs are closed under sufficiently complete ultraproducts. For this, we will prove that a certain functorial expansion of the $\mu$-AEC is axiomatizable by an infinitary logic (\emph{without} passing to an equivalent category):

\begin{defin}
  Let $\K$ be a $\mu$-AEC. The \emph{substructure functorial expansion} of $\K$ is the abstract class $\K^+$ defined as follows:

  \begin{enumerate}
  \item $\tau (\K^+) = \tau (\K) \cup \{P\}$, where $P$ is an $\LS (\K)$-ary predicate.
  \item $M^+ \in \K^+$ if and only if $M^+ \rest \tau (\K) \in \K$ and for any $\ba \in \fct{\LS (\K)}{M^+}$, $P^{M^+} (\ba)$ holds if and only if $\ran (\ba) \lea M^+ \rest \tau (\K)$, where we see $\ran (\ba)$ as a $\tau (\K)$-structure.
  \item For $M^+, N^+ \in \K^+$, $M^+ \leap{\K^+} N^+$ if and only if $M^+ \rest \tau (\K) \lea N^+ \rest \tau (\K)$.
  \end{enumerate}
\end{defin}

The substructure expansion is ``functorial'' in the sense of \cite[3.1]{sv-infinitary-stability-afml}: the reduct functor gives an isomorphism of concrete categories. The substructure functorial expansion has the property of having very simple morphisms:

\begin{thm}\label{substr-mc}
  Let $\K$ be a $\mu$-AEC and let $\K^+$ be its substructure functorial expansion. If $M^+, N^+ \in \K^+$ are such that $M^+ \subseteq N^+$, then $M^+ \leap{\K^+} N^+$.
\end{thm}
\begin{proof}
  For $M \in \K$, write $M^+$ for the expansion of $M$ to $\K^+$. Let $M, N \in \K$ and assume that $M^+ \subseteq N^+$. We have to see that $M \lea N$. For this, it is enough to show that for any $M_0 \lea M$ of cardinality at most $\LS (\K)$, we also have that $M_0 \lea N$ (indeed, we can then take the $\LS (\K)^+$-directed union of all such $M_0$'s). So let $M_0 \lea M$ have cardinality at most $\LS (\K)$: we must show that $M_0 \lea N$. Let $\ba$ be an enumeration of $M_0$. We have that $M^+ \models P[\ba]$ (where $P$ is the additional predicate in $\tau (\K)^+$), so $N^+ \models P[\ba]$ (as $M^+$ is a substructure of $N^+$). This means that $M_0 \lea N$, as desired.
\end{proof}

The substructure functorial expansion of a $\mu$-AEC can be axiomatized (a more complication variation of this, for AECs, is due to Baldwin and Boney \cite[3.9]{baldwin-boney}). Since the ordering is trivial by the previous result, this shows that any $\mu$-AEC is isomorphic (as a category) to the category of models of an $\Ll_{\infty, \infty}$ sentence, where the morphisms are injective homomorphisms.

\begin{thm}\label{axiom-mu-aec}
  Let $\K$ be a $\mu$-AEC and let $\K^+$ be its substructure functorial expansion. There is an $\Ll_{\left(2^{\LS (\K)}\right)^+, \LS (\K)^+}$ sentence $\phi$ such that $\K^+$ is the class of models of $\phi$. 
\end{thm}
\begin{proof}
  First note that for each $M_0 \in \K_{\le \LS (\K)}$, there is a sentence $\psi_{M_0} (\bx)$ of $\Ll_{\LS (\K)^+, \LS (\K)^+}$ coding its isomorphism type, i.e.\ whenever $M \models \psi_{M_0}[\ba]$, then $\ba$ is an enumeration of an isomorphic copy of $M_0$. Similarly, whenever $M_0, M_1$ are in $\K_{\le \LS (\K)}$ with $M_0 \lea M_1$, there is $\psi_{M_0, M_1} (\bx, \by)$ that codes that $(\bx, \by)$ is isomorphic to $(M_0, M_1)$ (so in particular $\bx \lea \by$). Let $S$ be a complete set of members of $\K_{\le \LS (\K)}$ (i.e.\ any other model is isomorphic to it) and let $T$ be a complete set of pairs $(M_0, M_1)$, with each in $\K_{\le \LS (\K)}$, such that $M_0 \lea M_1$. Now define the following:

  $$
  \phi_1 = \forall \bx \exists \by \left( \left( \bigvee_{M_0 \in S} \psi_{M_0} (\by)\right) \land \bx \subseteq \by \land P (\by) \right)
  $$

  $$
  \phi_2 = \forall \bx \forall \by \left( \left(\bx \subseteq \by \land P (\bx) \land P (\by)\right) \rightarrow \bigvee_{(M_0, M_1) \in T} \psi_{M_0, M_1} (\bx, \by)\right)
  $$

  $$\phi = \phi_1 \land \phi_2$$

  Where $\bx \subseteq \by$ abbreviates the obvious formula. This works. First, any $M^+ \in \K^+$ satisfies $\phi_1$ by the LST axiom and satisfies $\phi_2$ by the coherence axiom. Conversely, assume that $M^+ \models \phi$ and let $M := M^+ \rest \tau (\K)$. Consider the set:

  $$
  I = \{M_0 \in \K_{\le \LS (\K)} \mid U M_0 \subseteq U M, P^{M^+} (\bar{M_0})\}
  $$

  Where $\bar{M_0}$ refers to some enumeration of $M_0$. Then by construction of $\phi$, $I$ is a $\mu$-directed system in $\K$ and $\bigcup I = M$, so $M \in \K$. Similarly, $P^{M^+} (\ba)$ holds if and only if $\ran (\ba) \lea M$, so $M^+ \in \K^+$.
\end{proof}

\begin{cor}
  Any $\mu$-AEC $\K$ is closed under $\left(2^{\LS (\K)}\right)^+$-complete ultraproducts (in the sense that the appropriate generalization of \cite[4.3]{tamelc-jsl} holds).
\end{cor}
\begin{proof}
  By Theorems \ref{substr-mc} and \ref{axiom-mu-aec}, Łoś' theorem, and the fact that taking reducts commutes with ultraproducts.
\end{proof}

\section{On successor presentability ranks}

We start our study of the existence spectrum of an accessible category $\ck$: the set of regular cardinals $\lambda$ such that $\ck$ has an object of presentability rank $\lambda$. The goal is to say as much as possible by just looking at the accessibility spectrum: the set of cardinals $\lambda$ such that $\ck$ is $\lambda$-accessible.

In this section, we consider the question, first systematically investigated in \cite{beke-rosicky}, of whether the presentability rank of an object always has to be a successor (or, said differently, whether there can be objects of weakly inaccessible presentability rank). Assuming the accessibility spectrum is sufficiently large, we show there are \emph{no} objects of weakly inaccessible presentability rank, and explain how this generalizes previous results.

The following easy lemma characterizes existence in terms of the accessibility spectrum. It will also be used in the next section: 

\begin{lem}\label{basic-constr-lem}
  Let $\lambda$ be a regular cardinal and let $\ck$ be a category. The following are equivalent:

  \begin{enumerate}
  \item $\ck$ is $(\lambda, <\lambda)$-accessible.
  \item $\ck$ is $\lambda$-accessible and has \emph{no} objects of presentability rank $\lambda$.
  \end{enumerate}
\end{lem}
\begin{proof}
  Assume $\ck$ is $(\lambda, <\lambda)$-accessible. By definition, $\ck$ is clearly $\lambda$-accessible. If $M$ is a $\lambda$-presentable object of $\ck$, then it is a $\lambda$-directed colimit of $(<\lambda)$-presentable objects, hence by Fact \ref{sys-basic-facts} must itself be $(<\lambda)$-presentable.

  Conversely, if $\ck$ is $\lambda$-accessible and has no objects of presentability rank $\lambda$, then any $(\lambda, \lambda)$-system must be a $(\lambda, <\lambda)$-system, hence $\ck$ is $(\lambda, <\lambda)$-accessible.
\end{proof}

The following new result gives a criterion for $(\lambda, <\lambda)$-accessibility when $\lambda$ is weakly inaccessible.

\begin{thm}\label{wk-inacc-acc}
  If $\lambda$ is weakly inaccessible and $\ck$ is $(\mu, <\lambda)$-accessible for unboundedly-many $\mu < \lambda$, then $\ck$ is $(\lambda, <\lambda)$-accessible.
\end{thm}
\begin{proof}
  Let $S$ be the set of all regular cardinals $\mu < \lambda$ such that $\ck$ is $(\mu, <\lambda)$-accessible. Let $M$ be an object of $\ck$. For each $\mu \in S$, fix a $\mu$-directed system $\seq{M_i^\mu : i \in I_\mu}$, with maps $\seq{f_{i, j}^\mu: i \le j \in I_\mu}$ whose colimit is $M$ (with colimit maps $f_i^\mu$, $i \in I_\mu$). Let $I := \{(i, \mu) \mid \mu \in S, i \in I_\mu\}$. Order it by $(i, \mu_1) \le (j, \mu_2)$ if and only if $\mu_1 \le \mu_2$, and there exists a unique map $g: M_i^{\mu_1} \to M_j^{\mu_2}$ so that $f_i^{\mu_1} = f_{j}^{\mu_2} g$.

  Observe that $(I, \le)$ is a partial order. Also, if we fix $\mu \in S$ and $i \in I_\mu$, then $M_i^{\mu}$ is $(<\lambda)$-presentable, hence $\mu_1$-presentable for some regular $\mu_1 \in S$, $\mu \le \mu_1$. Thus for any $\mu_2 \in S$ with $\mu_1 \le \mu_2$, there exists $j \in I_{\mu_2}$ such that $(i, \mu) \le (j, \mu_2)$. 

  The last paragraph quickly implies that $I$ is $\lambda$-directed, and so the diagram induced by $I$ is the desired $(\lambda, <\lambda)$-system with colimit $M$.
\end{proof}
\begin{cor}\label{wk-inacc-acc-2}
  If $\lambda$ is weakly inaccessible and $\ck$ is $\mu$-accessible for unboundedly-many $\mu < \lambda$, then $\ck$ does not have an object of presentability rank $\lambda$. 
\end{cor}
\begin{proof}
  By Theorem \ref{wk-inacc-acc}, $\ck$ is $(\lambda, <\lambda)$-accessible. Now apply Lemma \ref{basic-constr-lem}.
\end{proof}

We obtain that any high-enough presentability rank in a well-accessible category (recall Definition \ref{acccat-defs}(\ref{well-acc-def})) must be a successor, which improves on \cite[4.2]{beke-rosicky} and \cite[5.5]{internal-sizes-jpaa}:

\begin{cor}\label{well-acc-succ}
  If $\ck$ is a well $\mu$-accessible category, then the presentability rank of any object that is not $\mu$-presentable must be a successor.
\end{cor}
\begin{proof}
  Immediate from Corollary \ref{wk-inacc-acc-2}.
\end{proof}

We have also recovered \cite[3.11]{internal-sizes-jpaa}:

\begin{cor}
  Let $\mu$ be a regular cardinal and let $\lambda > \mu$ be a weakly inaccessible cardinal. If $\ck$ is a $(\mu, <\lambda)$-accessible category and $\lambda$ is $\mu$-closed, then $\ck$ has no object of presentability rank $\lambda$.

  In particular, assuming $\ESCH$, high-enough presentability ranks are successors in any accessible category.
\end{cor}
\begin{proof}
  Since $\lambda$ is $\mu$-closed and limit, there are unboundedly-many regular $\theta \in [\mu, \lambda)$ that are $\mu$-closed. By Fact \ref{raising-acc}, for any such $\theta$, $\ck$ is $(\theta, <\lambda)$-accessible. By Theorem \ref{wk-inacc-acc}, $\ck$ is $(\lambda, <\lambda)$-accessible, hence by Lemma \ref{basic-constr-lem} cannot have an object of presentability rank $\lambda$.
\end{proof}

\section{The existence spectrum of a $\mu$-AEC}

We now refine a few results of \cite{internal-sizes-jpaa} concerning the existence spectrum of $\mu$-AECs, especially \cite[4.13]{internal-sizes-jpaa}.

We aim to study proper $(\lambda, <\lambda)$-systems---in the sense of Definition~\ref{systems-def}---and show that under certain conditions they do \emph{not} exist. This will give conditions under which an object of presentability rank $\lambda$ \emph{does} exist:

\begin{lem}\label{prop-sys-lem}
  Let $\lambda$ be a regular cardinal and let $\ck$ be a $\lambda$-accessible category. If $\ck$ has an object that is not $(<\lambda)$-presentable and $\ck$ has no proper $(\lambda, <\lambda)$-systems, then $\ck$ has an object of presentability rank $\lambda$.
\end{lem}
\begin{proof}
  By Lemma \ref{basic-constr-lem}, it suffices to show that $\ck$ is \emph{not} $(\lambda, <\lambda)$-accessible. Suppose for a contradiction that $\ck$ is $(\lambda, <\lambda)$-accessible. Let $M$ be an object that is not $(<\lambda)$-presentable. Then $M$ is the colimit of a $(\lambda, <\lambda)$-system, which must be proper because $M$ is not $(<\lambda)$-presentable (see Fact \ref{sys-basic-facts}), a contradiction to the assumption that there are no proper $(\lambda, <\lambda)$-systems.
\end{proof}

To help the reader, let us consider what a $(\lambda^+, <\lambda^+)$-system should be in an AEC $\K$ with $\lambda > \LS (\K)$. Since internal sizes correspond to cardinalities in that context (see Fact \ref{rank-bounds} or simply \cite[4.3]{lieberman-categ}), such a system must be a $\lambda^+$-directed system consisting of object of cardinality strictly less than $\lambda$. Because it is ``too directed,'' the system cannot be proper  (i.e.\ its colimit will just be a member of the system). We attempt here to generalize such an argument to suitable $\mu$-AECs. We will succeed when $\lambda$ is $\mu$-closed (Theorem \ref{no-proper-thm} --- notice that this is automatic when $\mu = \aleph_0$).

We will use the following key bound on the internal size of a subobject:

\begin{lem}\label{internal-monot}
  If $\K$ is a $\mu$-AEC, $M \lea N$ are in $\K$, $\lambda > \LS (\K)$ is a $\mu$-closed cardinal, and $N$ is $(<\lambda^+)$-presentable, then $M$ is $(<\lambda^+)$-presentable.
\end{lem}
\begin{proof}
  By Fact \ref{rank-bounds}, $|U N| < \lambda$. Of course, $|U M| \le |U N|$, so $|U M| < \lambda$ By Fact \ref{rank-bounds} again, $r_{\K} (M) \le |U M|^+ + \mu$, so $r_{\K} (M) \le \lambda + \mu = \lambda$, so $M$ is $(<\lambda^+)$-presentable, as desired.
\end{proof}

We require an additional refinement, concerning systems in which bounded subsystems have small colimits:

\begin{defin}\label{bounded-pres-def}
  A system $\seq{M_i : i \in I}$ in a given category is \emph{boundedly $(<\lambda)$-presentable} if whenever $I_0 \subseteq I$ is bounded in $I$, the colimit of $\seq{M_i : i \in I_0}$ is $(<\lambda)$-presentable (whenever it exists).
\end{defin}

\begin{lem}\label{bounded-pres}
  Let $\K$ be a $\mu$-AEC and let $\lambda > \LS (\K)^+$ be such that $\lambda^-$ is $\mu$-closed. Then any system in $\K$ consisting of $(<\lambda)$-presentable objects is boundedly $(<\lambda)$-presentable.
\end{lem}
\begin{proof}
  Let $\seq{M_i : i \in I}$ be a system consisting of $(<\lambda)$-presentable objects. Let $I_0 \subseteq I$ be bounded in $I$, say by $i$, and such that the colimit $M_{I_0}$ of the resulting system exists. We have that $M_{I_0} \lea M_{i}$. Since $\lambda^-$ is $\mu$-closed and $M_i$ is $(<\lambda)$-presentable, we can find $\lambda_0 \in [\LS (\K), \lambda)$ regular and $\mu$-closed such that $M_i$ is $\lambda_0$-presentable. By Lemma \ref{internal-monot}, $M_{I_0}$ is also $\lambda_0$-presentable, hence $(<\lambda)$-presentable, as desired.
\end{proof}

Using the bound of Lemma \ref{internal-monot} again, we now show that for most successor cardinals $\theta$, there are no proper $\theta$-directed boundedly $(<\theta)$-presentable systems:

\begin{lem}\label{bounded-pres-nonex}
  Let $\K$ be a $\mu$-AEC. If $\lambda > \LS (\K)$ is $\mu$-closed, then there are no proper $\lambda^+$-directed boundedly $(<\lambda^+)$-presentable systems.
\end{lem}
\begin{proof}
  Assume for a contradiction that $\seq{M_i : i \in I}$ is such a system, with colimit $M$. First, if $\lambda$ is regular, then using $\lambda^+$-directedness and properness we can find a chain $I_0 \subseteq I$ of type $\lambda$ such that $i < j$ in $I_0$ implies $M_i \neq M_j$. Then $\seq{M_i : i \in I_0}$ is proper, so its colimit (union) $M_{I_0}$ is not $\lambda$-presentable (Fact \ref{sys-basic-facts}). However, $I_0$ is bounded as $I$ is $\lambda^+$-directed, a contradiction to the hypothesis of bounded $(<\lambda^+)$-presentability.

  Assume now that $\lambda$ is singular. Let $\delta := \cf{\lambda}$ and write $\lambda = \sup_{\alpha < \delta} \lambda_\alpha$, with $\LS (\K) < \lambda_0$ and each $\lambda_\alpha$ regular and $\mu$-closed. Let $I_\alpha := \{i \in I \mid M_i \text{ is } \lambda_\alpha\text{-presentable}\}$. Note that $I = \bigcup_{\alpha < \delta} I_\alpha$.

  By Lemma \ref{cofin-dir}, there exists $\alpha < \delta$ such that $I_\alpha$ is cofinal in $I$. By renaming, we can assume without loss of generality that $\alpha = 0$: $I_0$ is already cofinal in $I$, hence $I_\alpha$ is cofinal in $I$ for all $\alpha < \delta$. Note that $I_0$ must itself be $\lambda^+$-directed. 

  Now pick $\seq{i_j : j < \lambda}$ an increasing sequence in $I_0$ such that $\seq{M_{i_j} : j < \lambda}$ is strictly increasing (this is possible by properness of the system). For $k < \lambda$ of cofinality at least $\mu$, let $N_k = \bigcup_{j < k} M_{i_j}$. Note that if $\alpha < \delta$ and $\cf{k} \ge \lambda_\alpha$, then by Fact \ref{sys-basic-facts}, $N_k$ is \emph{not} $\lambda_\alpha$-presentable. Fix $\alpha < \delta$ such that $\lambda_0 < \lambda_\alpha$. We then have that $N^1 := N_{\lambda_\alpha}$ is not $\lambda_\alpha$-presentable, but $N^2 := M_{i_{\lambda_{\alpha}}}$ is $\lambda_0$-presentable.  Since $N^1 \lea N^2$, this contradicts Lemma \ref{internal-monot}.
\end{proof}

We have arrived at the main technical result of this section.

\begin{thm}\label{no-proper-thm}
  Let $\K$ be a $\mu$-AEC. If $\lambda > \LS (\K)$ is $\mu$-closed, then there are no proper $(\lambda^+, <\lambda^+)$-systems in $\K$.
\end{thm}
\begin{proof}
  By Lemma \ref{bounded-pres} (with $\lambda^+$ in place of $\lambda$) and Lemma \ref{bounded-pres-nonex}.
\end{proof}

We get that, at least for successors of high-enough $\mu$-closed cardinals (or under SCH, see below), the presentability rank spectrum contains the accessibility spectrum.

\begin{cor}\label{spectrum-cor}
  Let $\K$ be a $\mu$-AEC. If $\lambda > \LS (\K)$ is a $\mu$-closed cardinal such that $\K$ is $\lambda^+$-accessible and $\K$ has an object of cardinality at least $\lambda$, then $\K$ has an object of presentability rank $\lambda^+$.
\end{cor}
\begin{proof}
  Let $M \in \K$ have cardinality at least $\lambda$. Using Fact \ref{rank-bounds} together with the assumption that $\lambda$ is $\mu$-closed, $M$ is not $(<\lambda^+)$-presentable. By Theorem \ref{no-proper-thm}, there are no proper $(\lambda^+, <\lambda^+)$-systems in $\K$. By Lemma \ref{prop-sys-lem}, $\K$ has an object of presentability rank $\lambda^+$.
\end{proof}

We have in particular recovered \cite[4.13]{internal-sizes-jpaa}. This will later be further generalized to any accessible category (Theorem \ref{no-mono-thm}).

\begin{cor}\label{recovered}
  Let $\K$ be a $\mu$-AEC and let $\lambda = \lambda^{<\mu} \ge \LS (\K)$. We have that $\K$ has an object of presentability rank $\lambda^+$ if at least one of the following conditions hold:

  \begin{enumerate}
  \item $\lambda > \LS (\K)$, $\lambda$ is $\mu$-closed, and $\K$ has an object of cardinality at least $\lambda$.
  \item $\lambda$ is regular and $\K$ has an object of cardinality at least $\lambda^+$.
  \end{enumerate}
\end{cor}
\begin{proof}
  Since $\lambda = \lambda^{<\mu}$, $\lambda^+$ is $\mu$-closed, so by Fact \ref{raising-acc}, $\ck$ is $\lambda^+$-accessible. Now:

  \begin{enumerate}
  \item This follows from Corollary \ref{spectrum-cor}.
  \item Since $\ck$ has $\mu$-directed colimits, it is also $(\lambda, \lambda^+)$-accessible. Since $\ck$ has an object of cardinality at least $\lambda^+$ and $\lambda^+$ is $\mu$-closed, Fact \ref{rank-bounds} (with $\lambda^+$ in place of $\lambda$) implies that this object is not $\lambda^+$-presentable. Now apply Fact \ref{succ-reg-existence}.
  \end{enumerate}
\end{proof}

\begin{remark}
  The example of well-orderings ordered by initial segment given in \cite[6.2]{internal-sizes-jpaa} shows that, even when $\mu = \aleph_0$, we may not have an object of rank $\LS (\K)^+$ if $\LS (\K)$ is singular.
\end{remark}

Under SCH, the statements simplify and we recover \cite[4.15]{internal-sizes-jpaa}:

\begin{cor}\label{sch-simplified-1}
Let $\K$ be a large $\mu$-AEC. If $\SCH$ holds above $\LS (\K)$, then for every $\lambda > \LS (\K)$ of cofinality at least $\mu$, $\K$ has an object of presentability rank $\lambda^+$. In particular, $\K$ is weakly LS-accessible.
\end{cor}
\begin{proof}
  By Fact \ref{sch-basic}, $\lambda = \lambda^{<\mu}$. If $\lambda$ is regular, we may apply Corollary \ref{recovered}(2). If $\lambda$ is singular, then by Fact \ref{sch-basic} it is $\mu$-closed so one may apply Corollary \ref{recovered}(1).
\end{proof}

Still under SCH, we obtain that the (successor) accessibility spectrum is eventually contained in the existence spectrum: 

\begin{cor}\label{acc-spectrum-cor}
  Assume $\ESCH$. Let $\ck$ be a large category with all morphisms monos. For all high-enough successor cardinals $\theta$, if $\ck$ is $\theta$-accessible, then $\ck$ has an object of presentability rank $\theta$.
\end{cor}
\begin{proof}
  By Fact \ref{maecsacc}, $\ck$ is equivalent to a $\mu$-AEC $\K$. By replacing $\K$ by a tail segment if necessary, we can assume without loss of generality that $\SCH$ holds above $\LS (\K)$. Pick $\theta > \LS (\K)^+$ successor such that $\K$ is $\theta$-accessible. Write $\theta = \lambda^+$. If $\cf{\lambda} \ge \mu$, Corollary \ref{sch-simplified-1} gives the result, so  we may assume that $\cf{\lambda} < \mu$. In this case, the SCH assumption implies that $\lambda$ is $\mu$-closed so we can apply Corollary \ref{spectrum-cor}.
\end{proof}

Recall (Definition \ref{ls-acc-def}) that a category is \emph{LS-accessible} if it has objects of all high-enough internal sizes. 

\begin{cor}\label{corwellls}
  Assuming $\ESCH$, any large well accessible category with all morphisms monos is LS-accessible.
\end{cor}

Note that Corollary \ref{corwellls} can be seen as a joint generalization of \cite[2.7]{ct-accessible-jsl} (LS-accessibility of large accessible categories with directed colimits and all morphisms monos) and \cite[5.9]{internal-sizes-jpaa} (LS-accessibility of large $\mu$-AECs with intersections): in both cases, the categories in question are well accessible with all morphisms monos (see Fact \ref{raising-acc}, \cite[5.4]{internal-sizes-jpaa}). Note however that the proof of LS-accessibility of large accessible categories with directed colimits and all morphisms monos does not assume SCH.

\section{The existence spectrum of an accessible category}\label{acc-sec}

A downside of the previous section was the assumption that all morphisms were monos. In the present section, we look at what can be said for arbitrary accessible categories. The main tool is the fact that the inclusion functor $\ck_{mono} \to \ck$ is (in a sense we make precise) accessible, hence plays reasonably well with internal sizes. While the notion of an accessible functor appears already in \cite[\S2.4]{makkai-pare}, we give here a more parameterized definition, in the style of Definition \ref{acccat-defs}: 

\begin{defin}
  Let $\lambda$ and $\mu$ be infinite cardinals, with $\mu$ regular. A functor $F: \ck \to \cL$ is \emph{$(\mu, <\lambda)$-accessible} if it preserves $\mu$-directed colimits and both $\ck$ and $\cL$ are $(\mu, <\lambda)$-accessible. We say that $F$ is \emph{$(\mu, \lambda)$-accessible} if $\lambda$ is regular and $F$ is $(\mu, <\lambda^+)$-accessible. We say that $F$ is \emph{$\mu$-accessible} precisely when it is $(\mu, \mu)$-accessible.
\end{defin}

\begin{fact}[{\cite[6.2]{indep-categ-advances}}]\label{mono-example}
  If $\ck$ is a $\mu$-accessible category, there there exists a cardinal $\lambda \ge \mu$ such that $\ck_{mono}$ is $(\mu, \lambda)$-accessible and moreover the inclusion functor $F$ of $\ck_{mono}$ into $\ck$ is $(\mu, \lambda)$-accessible.
\end{fact}

The following properties, describing the interaction of an accessible functor with presentability, were first systematically investigated in \cite[\S3]{beke-rosicky}:

\begin{defin}
  A functor $F: \ck \to \cL$ \emph{preserves $\lambda$-presentable objects} if whenever $M$ is $\lambda$-presentable in $\ck$, then $F (M)$ is $\lambda$-presentable (in $\cL$). We say that $F$ \emph{reflects $\lambda$-presentable objects} if $M$ is $\lambda$-presentable in $\ck$ whenever $F (M)$ is $\lambda$-presentable $\cL$. We also say that $F$ \emph{preserves $\lambda$-ranked objects} if whenever $r_{\ck} (M) = \lambda$, then $r_{\cL} (F M) = \lambda$. Similarly define what it means for $F$ to \emph{reflect $\lambda$-ranked objects}.
\end{defin}

Of course, there is a simple test to determine when a functor preserves rank given information as to whether it preserves and reflects presentable objects:

\begin{lem}\label{rank-pres-1}
  Let $F: \ck \to \cL$ be an accessible functor and let $\mu$ be a regular cardinal. If $F$ preserves $\mu$-presentable objects and reflects $(<\mu)$-presentable objects, then $F$ preserves $\mu$-ranked objects. Similarly, if $F$ preserves $(<\mu)$-presentable objects and reflects $\mu$-presentable objects, then $F$ reflects $\mu$-ranked objects.
\end{lem}
\begin{proof}
  We prove the first statement (the proof of the second is similar). Let $M$ be an object of $\ck$ such that $r (M)  = \mu$. Then $r (F M) \le \mu$ because $F$ preserves $\mu$-presentable objects, and if $r (F M) < \mu$, then $r (M) < \mu$ because $F$ reflects $(<\mu)$-presentable objects, contradiction. Thus $r (M) = r (F M)$. 
\end{proof}

For a functor to \emph{reflect} $\lambda$-presentable objects, it enough that it is sufficiently accessible and that the functor \emph{reflects split epimorphisms} (i.e.\ if $F f$ is a split epi, then $f$ is a split epi). This was isolated in \cite[3.6]{beke-rosicky}. We now proceed to mine the proof of this result to extract what can be said in our more parameterized setup: 

\begin{lem}\label{proper-split}
  Let $\lambda$ and $\mu$ be cardinals, $\mu$ regular. Let $F: \ck \to \cL$ be a functor reflecting split epimorphisms and preserving $\mu$-directed colimits. Let $\seq{M_i : i \in I}$ be a $(\mu, <\lambda)$-system with colimit $M$. If $\seq{M_i : i \in I}$ is proper, then $\seq{F M_i : i \in I}$ is proper. If in addition $F$ preserves $(<\lambda)$-presentable objects, then $\seq{F M_i : i \in I}$ is a $(\mu, <\lambda)$-system.
\end{lem}
\begin{proof}
  Since $F$ preserves $\mu$-directed colimits, the colimit of $\seq{F M_i : i \in I}$ is $F M$. Suppose that the identity map on $F M$ factors through some $F M_i$, via a map $g: F M \to F M_i$. That is, $(F f_i) g = \id_{F M}$, where $f_i : M_i \to M$ is a colimit map. Then $F f_i$ is a split epimorphism, hence $f_i$ is a split epimorphism, i.e.\ $f_i g = \id_M$, so $\seq{M_i :i \in I}$ is not proper. The last sentence is immediate from the definition.
\end{proof}

\begin{fact}[{\cite[3.6]{beke-rosicky}}]\label{functor-refl}
  Let $\lambda$ be an uncountable cardinal, and let $F: \ck \to \cL$ be a functor reflecting split epimorphisms. If there exists a regular cardinal $\lambda_0 < \lambda$ such that $F$ is $(\lambda_0, <\lambda)$-accessible and for all regular cardinals $\mu \in [\lambda_0, \lambda)$, $\ck$ is $(\mu, <\lambda)$-accessible, then $F$ reflects $(<\lambda)$-presentable objects.
\end{fact}
\begin{proof}
  Assume that $F M$ is $(<\lambda)$-presentable. Pick a regular cardinal $\mu \in [\lambda_0, \lambda)$ such that $F M$ is $\mu$-presentable. By $(\mu, <\lambda)$-accessibility, $M$ is the colimit of a $(\mu, <\lambda)$-system $\seq{M_i : i \in I}$. Assume for a contradiction that $M$ is \emph{not} $(<\lambda)$-presentable. Then $\seq{M_i : i \in I}$ must be proper by Fact \ref{sys-basic-facts}(\ref{sys-basic-facts-2}). By Lemma \ref{proper-split}, $\seq{F M_i : i \in I}$ is proper, hence its colimit $F M$ cannot be $\mu$-presentable by Fact \ref{sys-basic-facts}(\ref{sys-basic-facts-2}), a contradiction.
\end{proof}

In passing, we can deduce the following powerful criterion for existence of an object of regular internal size. Notice that this is a generalization of Fact \ref{succ-reg-existence} (which is the special case of the identity functor).

\begin{thm}\label{succ-reg-functor}
    Let $\lambda$ be a regular cardinal and let $F: \ck \to \cL$ be a $(\lambda, \lambda^+)$-accessible functor that preserves $\lambda^+$-presentable objects and reflects isomorphisms. If all morphisms in the image of $F$ are monos and $\ck$ has an object that is not $\lambda^+$-presentable, then $\cL$ has an object of presentability rank $\lambda^+$.
\end{thm}
\begin{proof}
  Since $F$ reflects isomorphisms and all morphisms in the image of $F$ are monos, $F$ reflects split epimorphisms. By Fact \ref{succ-reg-existence}, $\ck$ has a proper $(\lambda, \lambda^+)$-system $\seq{M_i : i \in I}$ with $\lambda$-many objects. By Lemma \ref{proper-split}, $\seq{F M_i : i \in I}$ is a proper $(\lambda, \lambda^+)$-system. By Corollary \ref{succ-reg-cor}, the colimit of this system in $\cL$ has presentability rank $\lambda^+$.
\end{proof}

To \emph{preserve} $\lambda$-presentable objects, a cardinal arithmetic assumption on $\lambda$ suffices:

\begin{fact}\label{functor-pres}
  Let $F: \ck \to \cL$ be a $(\mu, \lambda_0)$-accessible functor. Let $\lambda_1 \ge \lambda_0$ be a regular cardinal such that the image of any $\lambda_0$-presentable object is $\lambda_1$-presentable. If $\lambda > \lambda_1$ is such that $\lambda^-$ is $\mu$-closed, then $F$ preserves $(<\lambda)$-presentable objects.
\end{fact}
\begin{proof}
  This is similar to the proof of \cite[2.19]{adamek-rosicky}. We include some details for the convenience of the reader. Let $A$ be a $(<\lambda)$-presentable object of $\ck$. Pick a regular $\theta_0 < \lambda$ such that $A$ is $\theta_0$-presentable. If $\lambda$ is a successor and $\lambda^-$ is regular, let $\theta := \lambda^-$. Otherwise, $\lambda^-$ is limit and we let $\theta := \left(\left(\theta_0 + \lambda_1\right)^{<\mu}\right)^+$. In either case, we have that $\lambda_1 \le \theta < \lambda$ (because $\lambda^-$ is $\mu$-closed), $\theta$ is regular, $\theta$ is $\mu$-closed (by \cite[2.13(5)]{adamek-rosicky}), and $A$ is $\theta$-presentable. By the proof of \cite[2.3.10]{makkai-pare} (see \cite[2.15]{adamek-rosicky}), $A$ is a $\theta$-directed colimit of a diagram consisting of $\mu$-directed colimits of fewer than $\theta$-many $\lambda_0$-presentable objects. Since $A$ is $\theta$-presentable, $A$ is a retract of such a $\mu$-directed colimit. Since $F$ preserves $\mu$-directed colimits and any functor preserves retractions, $F A$ is a retract of a $\mu$-directed colimit of fewer than $\theta$-many $\lambda_1$-presentable objects. By Fact \ref{sys-basic-facts}(\ref{sys-basic-facts-1}), $F A$ is $\theta$-presentable, hence $(<\lambda)$-presentable.
\end{proof}
\begin{remark}
  Instead of $\lambda^-$ $\mu$-closed, it suffices to assume that $\mu \tlt \lambda^-$ (see Definition \ref{sch-def}(\ref{tlt-def}) and Remark \ref{tlt-rmk}). Conversely, if $\mu < \lambda$ are cardinals such that $\mu$ is regular and every $\mu$-accessible functor which preserves $\mu$-presentable objects also preserves $(<\lambda)$-presentable objects, then $\mu \tlt \lambda^-$ (consider the functor $A \mapsto [A]^{<\mu}$ from the category of sets to the category of $\mu$-directed posets).
\end{remark}

We obtain the following cardinal arithmetic test for preservation and reflection of ranks:

\begin{lem}\label{rank-pres-2}
  Let $F: \ck \to \cL$ be a $(\mu, \lambda_0)$-accessible functor that preserves $\lambda_0$-presentable objects and reflects split epimorphisms. 

  If $\theta > \lambda_0$ is the successor of a $\mu$-closed cardinal of cofinality at least $\mu$, then $F$ preserves and reflects $\theta$-ranked objects.
\end{lem}
\begin{proof}
  Write $\theta = \lambda^+$, with $\lambda$ a $\mu$-closed cardinal of cofinality at least $\mu$. Note that $\lambda^{<\mu} = \lambda$, since it has cofinality at least $\mu$. Thus both $\lambda$ and $\theta$ are $\mu$-closed. This implies that $F$ preserves $\theta$-presentable objects and preserves $(<\theta)$-presentable objects (Fact \ref{functor-pres}). We also have that $F$ reflects $(<\theta)$-presentable objects and $F$ reflects $\theta$-presentable objects (use Facts \ref{raising-acc} and \ref{functor-refl}). Now apply Lemma \ref{rank-pres-1}.
\end{proof}

Using Corollary \ref{recovered}, we obtain the following existence spectrum result if the domain of the functor is a large $\mu$-AEC:

\begin{lem}\label{mu-aec-func}
  Let $\K$ be a $\mu$-AEC, let $\lambda_0 > \LS (\K)$ be regular, and let $F: \K \to \cL$ be a $(\mu, \lambda_0)$-accessible functor that preserves $\lambda_0$-presentable objects and reflects isomorphisms. Let $\lambda \ge \lambda_0$ be a $\mu$-closed cardinal of cofinality at least $\mu$. If $\K$ has an object of cardinality at least $\lambda$, then $\cL$ has an object of presentability rank $\lambda^+$.
\end{lem}
\begin{proof}
  Since all morphisms of $\K$ are monos, $F$ reflects split epimorphisms. Since $\lambda$ is $\mu$-closed and has cofinality at least $\mu$, $\lambda = \lambda^{<\mu}$. By Corollary \ref{recovered}, $\K$ has an object $M$ of presentability rank $\lambda^+$. By Lemma \ref{rank-pres-2} (where $\theta$ there stand for $\lambda^+$ here), $F$ preserves $\lambda^+$-ranked objects, so $F M$ has presentability rank $\lambda^+$.
\end{proof}

Putting all the results together, we obtain an existence spectrum result for any large accessible category. This extends Corollary \ref{recovered}.

\begin{thm}\label{no-mono-thm}
  Let $\ck$ be a large $\mu$-accessible category.

  \begin{enumerate}
  \item For every high-enough regular $\lambda$ such that $\lambda = \lambda^{<\mu}$, $\ck$ has an object of presentability rank $\lambda^+$.
  \item There exists a regular cardinal $\mu'$ such that for every high-enough $\mu'$-closed cardinal $\lambda$ of cofinality at least $\mu'$, $\ck$ has an object of presentability rank $\lambda^+$.
  \end{enumerate}
\end{thm}
\begin{proof} \
  \begin{enumerate}
  \item  By Fact \ref{mono-example}, there exists a cardinal $\lambda_0$ such that the inclusion functor $F$ of $\ck_{mono}$ into $\ck$ is $(\mu, \lambda_0)$-accessible. Of course, $F$ also reflects isomorphisms. Let $\lambda > \lambda_0$ be regular such that $\lambda = \lambda^{<\mu}$. By Fact \ref{functor-pres}, $F$ preserves $\lambda^+$-presentable objects and by Fact \ref{raising-acc}, $F$ is $(\lambda, \lambda^+)$-accessible. Since $\ck$ is large, $\ck_{mono}$ is also large, so by Theorem \ref{succ-reg-functor}, $\ck$ has an object of presentability rank $\lambda^+$.
  \item As before, $\ck_{mono}$ is an accessible category with all morphisms monos, so by Fact \ref{maecsacc}, it is equivalent to a $\mu'$-AEC $\K^\ast$, for some regular cardinal $\mu'$. Let $F: \K^\ast \to \ck$ be the composition of the equivalence with the inclusion of $\ck_{mono}$ into $\ck$. Then $F$ is $(\mu', \lambda_0)$-accessible, for some regular cardinal $\lambda_0 > \LS (\K)$, and $F$ reflects isomorphisms. Taking $\lambda_0$ bigger if needed (and using Fact \ref{functor-pres}), we can assume without loss of generality that $F$ also preserves $\lambda_0$-presentable objects. Let $\lambda \ge \lambda_0$ be a $\mu'$-closed cardinal of cofinality at least $\mu'$. Since $\K$ is large, Lemma \ref{mu-aec-func} applies and so $\cL$ has an object of presentability rank $\lambda^+$.
  \end{enumerate}
\end{proof}

We obtain the main result of this section. This extends for example \cite[A.2]{internal-sizes-jpaa} --- weak LS-accessibility of large locally multipresentable categories --- at the cost of $\ESCH$:

\begin{cor}\label{esch-no-mono}
  Assuming $\ESCH$, any large accessible category has objects of all internal sizes of high-enough cofinality. In particular, any large accessible category is weakly LS-accessible.
\end{cor}
\begin{proof}
  Let $\ck$ be a large $\mu$-accessible category, and let $\mu'$ be as given by Theorem \ref{no-mono-thm}. Let $\lambda$ be a high-enough cardinal of cofinality at least $\mu'$ (the proof will give how big we need to take it). If $\lambda$ is a regular cardinal, then (by $\ESCH$, see Fact \ref{sch-basic}) $\lambda = \lambda^{<\mu}$, so by Theorem \ref{no-mono-thm}, $\ck$ has an object of presentability rank $\lambda^+$. Assume now that $\lambda$ is a singular cardinal. Then $\lambda$ is in particular a limit cardinal, so (by $\ESCH$) $\lambda$ is $\mu'$-closed. By Theorem \ref{no-mono-thm}, $\ck$ has an object of presentability rank $\lambda^+$.
\end{proof}

\section{Filtrations}\label{secresolv}

We consider conditions under which, in a general category, we can ensure than any object is not merely the colimit of an appropriately directed system of objects of strictly smaller internal size, but rather the colimit of a \emph{chain} of such objects.  The existence of such \emph{filtrations} (sometimes also called \emph{resolutions}) is crucial to a host of model-theoretic constructions, and should be of considerable use in the further development of classification theory at the present level of generality.

\begin{defin}
  For $\mu$ a regular cardinal and $\lambda$ an infinite cardinal, a \emph{$(\mu, <\lambda)$-chain} (in a category $\ck$) is a diagram $\seq{M_i : i < \mu}$ indexed by $\mu$, all of whose objects are $(<\lambda)$-presentable. We call $\mu$ the \emph{length} of the chain. A \emph{$(<\lambda)$-chain} is a $(\mu, <\lambda)$-chain for some regular $\mu < \lambda$. For $\lambda$ a regular cardinal, a \emph{$\lambda$-chain} is a $(<\lambda^+)$-chain. 

  For $\theta$ a regular cardinal, we say that a chain $\seq{M_i : i < \mu}$ is \emph{$\theta$-smooth} if for every $i < \mu$ of cofinality at least $\theta$, $M_i$ is the colimit of $\seq{M_j : j < i}$.
\end{defin}

Note that $(\mu, <\lambda)$-chains are $(\mu, <\lambda)$-systems in the sense of Definition \ref{systems-def}. We will use the terminology of systems introduced in the preliminaries. The reader may also wonder why we are looking only at chains indexed by a regular cardinal. This is because any system $\seq{M_i : i \in I}$ indexed by a linear order $I$ has a cofinal subsystem of the form $\seq{M_{i_j} : i_j < \mu}$, where $\mu$ is the cofinality of $I$.

The next definition is the object of study of this section: 

\begin{defin}
  Let $\ck$ be a category. A \emph{filtration} of an object $M$ is a $(<r_{\ck} (M))$-chain with colimit $M$. We call $M$ \emph{filtrable} if it (has a presentability rank and) has a filtration.
\end{defin}

Observe that the length of a filtration is determined by the presentability rank:

\begin{lem}\label{resol-length}
  Let $\lambda$ be a regular cardinal and let $\ck$ be a category. If there exists a $(<\lambda)$-chain whose colimit is not $(<\lambda)$-presentable, then $\lambda$ is a successor and the chain must have length $\cf{\lambda^-}$. In particular, any filtrable object $M$ has successor presentability rank and any of its filtrations will have length $\cf{|M|_{\ck}}$.
\end{lem}
\begin{proof}
  Let $\mu < \lambda$ be a regular cardinal and let $\seq{M_i : i < \mu}$ be a $(\mu, <\lambda)$-chain in $\ck$ with a colimit $M$ that is not $(<\lambda)$-presentable. By Fact \ref{sys-basic-facts}(\ref{sys-basic-facts-2}), the chain is proper. Next, assume for a contradiction that $\lambda$ is weakly inaccessible. By Fact \ref{sys-basic-facts}(\ref{sys-basic-facts-1}), $M$ is $(<\lambda)$-presentable, a contradiction to the fact that it has presentability rank $\lambda$. This shows that $\lambda$ is a successor. Let $\lambda_0 = \lambda^-$. We now have to see that $\cf{\lambda_0} = \mu$. We consider two cases depending on whether $\lambda_0$ is regular or singular:

  \begin{itemize}
  \item If $\lambda_0$ is regular, then $\seq{M_i : i < \mu}$ is a $(\mu, \lambda_0)$-system. Since $\mu < \lambda$, we know that $\mu \le \lambda_0$. If $\mu < \lambda_0$, then by Fact \ref{sys-basic-facts}(\ref{sys-basic-facts-1}), $M$ would be $(\mu^+ + \lambda_0)$-presentable, hence $\lambda_0$-presentable, contradicting that it has presentability rank $\lambda$. Thus $\mu = \lambda_0 = \cf{\lambda_0}$.
  \item If $\lambda_0$ is singular, then $\seq{M_i : i < \mu}$ is a $(\mu, <\lambda_0)$-system, and moreover (because $\mu$ is regular) $\mu \neq \lambda_0$, so $\mu^+ < \lambda_0$. If $\cf{\lambda_0} > \mu$, then there exists a regular $\lambda_1 < \lambda_0$ such that $\seq{M_i : i < \mu}$ is a $(\mu, \lambda_1)$-system. By Fact \ref{sys-basic-facts}(\ref{sys-basic-facts-1}), $M$ is $(\mu^+ + \lambda_1)$-presentable, hence $(<\lambda)$-presentable, a contradiction.

    Thus $\cf{\lambda_0} \le \mu$. If $\cf{\lambda_0} < \mu$, let $\seq{\theta_\alpha : \alpha < \cf{\lambda_0}}$ be an increasing chain of regular cardinals cofinal in $\lambda_0$. For $\alpha < \cf{\lambda_0}$, let $I_\alpha := \{i < \mu \mid M_i \text{ is } \theta_\alpha\text{-presentable}\}$. By Lemma \ref{cofin-dir} (applied to $I = \mu$), there exists $\alpha < \cf{\lambda_0}$ so that $I_\alpha$ is cofinal in $\mu$. In particular, $M$ is still the colimit of $\seq{M_i : i \in I_\alpha}$. The latter system is a $(\mu, \theta_\alpha)$-system, hence (again by Fact \ref{sys-basic-facts}(\ref{sys-basic-facts-1})) $M$ is $(\mu^+ + \theta_\alpha)$-presentable, so $(<\lambda)$-presentable, a contradiction. The only remaining possibility is that $\cf{\lambda_0} = \mu$, which is what we wanted to prove.
  \end{itemize}
\end{proof}

It follows that if the category has directed colimits, we can take the filtration to be smooth. More generally:

\begin{lem}\label{smooth-lem}
  Any filtration of an object $M$ of presentability rank $\lambda$ is boundedly $(<\lambda)$-presentable (Definition \ref{bounded-pres-def}). In particular, if $\theta < \lambda$ is regular such that $\ck$ has $\theta$-directed colimits, then $M$ has a $\theta$-smooth filtration.
\end{lem}
\begin{proof}
  Let $\mu < \lambda$ be regular, and let $\seq{M_i : i < \mu}$ be a $(<\lambda)$-chain whose colimit is $M$. By Lemma \ref{resol-length}, $\lambda$ is a successor cardinal and $\mu = \cf{\lambda^-}$. Let $\delta < \mu$ be a limit ordinal. We will show that the colimit of $\seq{M_i : i < \delta}$ (assuming it exists) is $(<\lambda)$-presentable. The ``in particular'' part will then follow, since, when $\cf{\delta} \ge \theta$, it suffices to replace $M_\delta$ by the colimit of $\seq{M_i : i < \delta}$.

  Note that $\delta < \mu \le \lambda^-$, so $\delta^+ < \lambda$. By cofinality considerations, there exists a regular cardinal $\lambda_0 < \lambda$ such that for all $i < \delta$, $M_i$ is $\lambda_0$-presentable. By Fact \ref{sys-basic-facts}(\ref{sys-basic-facts-1}), we get that the colimit of $\seq{M_i : i < \delta}$ is $(\delta^+ + \lambda_0)$-presentable, hence $(<\lambda)$-presentable, as desired.  
\end{proof}

Using the definition of presentability, it is also easy to generalize the well known facts that, for objects of regular cardinality (that is, of presentability rank the successor of a regular cardinal), any two smooth filtrations are the same on a \emph{club}: a \emph{closed unbounded} set of indices. See for example \cite[6.11]{mu-aec-jpaa}.

We give a name to categories where every object of a given presentability rank is filtrable:

\begin{defin}\label{defresolving}
  For a regular cardinal $\lambda$, we say an accessible category $\ck$ is \emph{$\lambda$-filtrable} if any object of presentability rank $\lambda$ is filtrable. For a regular cardinal $\mu$, we say that $\ck$ is \emph{well $\mu$-filtrable} if it is $\lambda$-filtrable for any regular $\lambda \ge \mu$. We say $\ck$ is \emph{well filtrable} if it is well $\mu$-filtrable for some regular cardinal $\mu$.
  
  Similarly, we say that $\ck$ is \emph{almost $\lambda$-filtrable} if any object $M$ of presentability rank $\lambda$ is a retract of a filtrable $\lambda$-presentable object $N$ (i.e.\ there exists a split epimorphisms from $N$ to $M$). Define \emph{almost well $\lambda$-filtrable} and \emph{almost well filtrable} as expected.
\end{defin}

\begin{remark}
  The technical notion of being \emph{almost} filtrable is included here because we do not know whether retracts of filtrable objects are filtrable. Of course, in categories where all morphisms are monos, retractions are just isomorphisms and so this technical distinction is irrelevant.
\end{remark}

In the rest of this section, we give some conditions implying existence of filtrations. The next result is our main tool: 

\begin{lem}\label{filtr-lem}
  Let $\ck$ be a category, let $\mu \le \theta \le \lambda$ be given with $\mu$ regular, $\theta$ regular, and $\mu \tlt \lambda$ (see Definition \ref{sch-def}(\ref{tlt-def})). Let $A \in \ck$ be an object of presentability rank $\lambda^+$. If $\ck$ has $\mu$-directed colimits and $A$ is the colimit of a $(\mu, \theta)$-system of cardinality at most $\lambda$, then $A$ is filtrable.
\end{lem}
\begin{proof}
  Let $D: I \to \ck$ be a $(\mu, \theta)$-system with $|I| \le \lambda$. It is a direct consequence of the definition of the $\tlt$ relation that we can find $\seq{I_i : i < \lambda}$ an increasing chain of $\mu$-directed subposets of $I$ such that $\bigcup_{i < \lambda} I_i = I$ and for all $i < \lambda$, $|I_i| <  \lambda_i$ for some regular $\lambda_i \in [\theta, \lambda]$. By Fact \ref{sys-basic-facts}(\ref{sys-basic-facts-1}), for each $i < \lambda$ we have that $A_i := \text{colim} (D \rest I_i)$ is $\lambda_i$-presentable, hence $(<\lambda^+)$-presentable, so the system consisting of $\seq{A_i : i < \lambda}$ is the desired chain. 
\end{proof}

We deduce the main theorem of this section. Note that the third part is an improvement on \cite[Lemma 1]{rosicky-sat-jsl}, which assumed in addition existence of directed colimits.

\begin{thm}\label{filtr-thm} Let $\mu$ be a regular cardinal, and let $\lambda$ be a cardinal such that $\mu \tlt \lambda$ and\footnote{If $\lambda$ is regular this follows, but not necessarily if $\lambda$ is singular (e.g.\ if $\lambda$ is strong limit).} $\mu \tlt \lambda^+$.
  
  \begin{enumerate}
  \item If $\ck$ is a category which is $(\mu, \theta)$-accessible for some regular $\theta \le \lambda$, then $\ck$ is almost $\lambda^+$-filtrable.
  \item Any $\mu$-accessible category is $\lambda^+$-filtrable.
  \item Any $\mu$-accessible category is $\mu^+$-filtrable.
  \end{enumerate}
\end{thm}
\begin{proof} \
  \begin{enumerate}
  \item Fix $\theta \le \lambda$ regular such that $\ck$ is $(\mu, \theta)$-accessible. By definition of $\mu \tlt \lambda$, we can increase $\theta$ if necessary to assume without loss of generality that $\mu \tlt \theta$. Let $A$ be an object of presentability rank $\lambda^+$. By the moreover part of Fact \ref{raising-acc}, $A$ is a $\lambda^+$-directed colimit of a diagram whose objects are all $\mu$-directed colimits of at most $\lambda$-many $\theta$-presentables. Thus $A$ is a retract of such a colimit, and the result follows from Lemma \ref{filtr-lem}.
  \item As before, but this time, by \cite[2.3.11]{makkai-pare}, we actually have that $A$ is a $\mu$-directed colimits of at most $\lambda$-many $\mu$-presentable objects (not just a retract of such a colimit).
  \item Let $A$ be an object of presentability rank $\mu^+$. Since $\mu \tlt \mu^+$, by \cite[2.3.11]{makkai-pare}, $A$ is a the colimit of a $\mu$-directed diagram $D: I \to \ck$ containing at most $\mu$-many $\mu$-presentable objects. Enumerate $I$ as $\{i_\alpha : \alpha < \mu\}$. We build $\seq{j_\alpha : \alpha < \mu}$ an increasing sequence in $I$ such that $j_\alpha \ge i_\beta$ whenever $\beta < \alpha$. This is possible because $I$ is $\mu$-directed, and in the end we get that $\seq{D_{j_\alpha} : \alpha < \mu}$ is the desired chain.
  \end{enumerate}
\end{proof}

We deduce some results on accessible categories with directed colimits. Note in particular that we have recovered the well known fact that any AEC $\K$ is well $\LS (\K)^{++}$-filtrable.

\begin{cor}\label{filtr-cor} \
  \begin{enumerate}
  \item Any finitely accessible category is well $\aleph_1$-filtrable.
  \item Any $\lambda$-accessible category with directed colimits is almost well $\lambda^+$-filtrable.
  \item Any $\lambda$-accessible category with directed colimits and all morphisms monos is well $\lambda^+$-filtrable.
  \end{enumerate}
\end{cor}
\begin{proof}
  The first two results are immediate by setting $\mu := \aleph_0$ in Theorem \ref{filtr-thm}. For the third, simply observe that a retraction which is a monomorphism is an isomorphism.
\end{proof}

\begin{remark}
  Using a fat small object argument to eliminate retracts (see \cite{fat-small-obj}), we can also show that any locally $\lambda$-presentable category is well $\lambda^+$-filtrable.
\end{remark}

\bibliographystyle{amsalpha}
\bibliography{internal-size-improved}

\newcommand{\etalchar}[1]{$^{#1}$}
\providecommand{\bysame}{\leavevmode\hbox to3em{\hrulefill}\thinspace}
\providecommand{\MR}{\relax\ifhmode\unskip\space\fi MR }
\providecommand{\MRhref}[2]{%
  \href{http://www.ams.org/mathscinet-getitem?mr=#1}{#2}
}
\providecommand{\href}[2]{#2}
\begin{thebibliography}{BGL{\etalchar{+}}16}

\bibitem[AHS04]{joy-of-cats}
Jiř{\'{\i}} Ad{\'{a}}mek, Horst Herrlich, and George~E. Strecker,
  \emph{Abstract and concrete categories}, online edition ed., 2004, Available
  from \url{http://katmat.math.uni-bremen.de/acc/}.

\bibitem[AR94]{adamek-rosicky}
Jiř{\'{\i}} Ad{\'a}mek and Jiř{\'{\i}} Rosický, \emph{Locally presentable
  and accessible categories}, London Math. Society Lecture Notes, Cambridge
  University Press, 1994.

\bibitem[BB17]{baldwin-boney}
John~T. Baldwin and Will Boney, \emph{Hanf numbers and presentation theorems in
  {A}{E}{C}s}, Beyond first order model theory (Jos{\'e} Iovino, ed.), CRC
  Press, 2017, pp.~327--352.

\bibitem[BGL{\etalchar{+}}16]{mu-aec-jpaa}
Will Boney, Rami Grossberg, Michael~J. Lieberman, Jiř{\'{\i}} Rosický, and
  Sebastien Vasey, \emph{$\mu$-{A}bstract elementary classes and other
  generalizations}, The Journal of Pure and Applied Algebra \textbf{220}
  (2016), no.~9, 3048--3066.

\bibitem[Bon14]{tamelc-jsl}
Will Boney, \emph{Tameness from large cardinal axioms}, The Journal of Symbolic
  Logic \textbf{79} (2014), no.~4, 1092--1119.

\bibitem[BR12]{beke-rosicky}
Tibor Beke and Jiř{\'{\i}} Rosický, \emph{Abstract elementary classes and
  accessible categories}, Annals of Pure and Applied Logic \textbf{163} (2012),
  2008--2017.

\bibitem[BY05]{bymcats}
Itay Ben-Yaacov, \emph{Uncountable dense categoricity in cats}, The Journal of
  Symbolic Logic \textbf{70} (2005), no.~3, 829--860.

\bibitem[HH09]{maec-hihy}
{\AA}sa Hirvonen and Tapani Hyttinen, \emph{Categoricity in homogeneous
  complete metric spaces}, Archive for Mathematical Logic \textbf{48} (2009),
  269--322.

\bibitem[Jec03]{jechbook}
Thomas Jech, \emph{Set theory}, 3rd ed., Springer-Verlag, 2003.

\bibitem[Lie11]{lieberman-categ}
Michael~J. Lieberman, \emph{Category-theoretic aspects of abstract elementary
  classes}, Annals of Pure and Applied Logic \textbf{162} (2011), no.~11,
  903--915.

\bibitem[LR16]{ct-accessible-jsl}
Michael~J. Lieberman and Jiř{\'{\i}} Rosický, \emph{Classification theory for
  accessible categories}, The Journal of Symbolic Logic \textbf{81} (2016),
  no.~1, 151--165.

\bibitem[LR17]{lrcaec-jsl}
\bysame, \emph{Metric abstract elementary classes as accessible categories},
  The Journal of Symbolic Logic \textbf{82} (2017), no.~3, 1022--1040.

\bibitem[LRV]{more-indep-v2}
Michael~J. Lieberman, Jiř{\'{\i}} Rosický, and Sebastien Vasey, \emph{Weak
  factorization systems and stable independence}, Preprint. URL:
  \url{https://arxiv.org/abs/1904.05691v2}.

\bibitem[LRV19a]{indep-categ-advances}
\bysame, \emph{Forking independence from the categorical point of view},
  Advances in Mathematics \textbf{346} (2019), 719--772.

\bibitem[LRV19b]{internal-sizes-jpaa}
\bysame, \emph{Internal sizes in $\mu$-abstract elementary classes}, Journal of
  Pure and Applied Algebra \textbf{223} (2019), no.~10, 4560--4582.

\bibitem[LRV19c]{multipres-pams}
\bysame, \emph{Universal abstract elementary classes and locally
  multipresentable categories}, Proceedings of the American Mathematical
  Society \textbf{147} (2019), no.~3, 1283--1298.

\bibitem[Lur09]{htt-lurie}
Jacob Lurie, \emph{Higher topos theory}, Annals of Mathematics Studies, no.
  170, Princeton University Press, 2009.

\bibitem[MP89]{makkai-pare}
Michael Makkai and Robert Par{\'e}, \emph{Accessible categories: The
  foundations of categorical model theory}, Contemporary Mathematics, vol. 104,
  American Mathematical Society, 1989.

\bibitem[MRV14]{fat-small-obj}
Michael Makkai, Jiř{\'{\i}} Rosický, and Luk{\'a}{\v s} Vokř{\'{\i}}nek,
  \emph{On a fat small object argument}, Advances in Mathematics \textbf{254}
  (2014), 49--68.

\bibitem[Ros97]{rosicky-sat-jsl}
Jiř{\'{\i}} Rosický, \emph{Accessible categories, saturation and
  categoricity}, The Journal of Symbolic Logic \textbf{62} (1997), no.~3,
  891--901.

\bibitem[She87]{sh88}
Saharon Shelah, \emph{Classification of non elementary classes {II}. {A}bstract
  elementary classes}, Classification Theory (Chicago, IL, 1985) (John~T.
  Baldwin, ed.), Lecture Notes in Mathematics, vol. 1292, Springer-Verlag,
  1987, pp.~419--497.

\bibitem[SV]{multidim-v2}
Saharon Shelah and Sebastien Vasey, \emph{Categoricity and multidimensional
  diagrams}, Preprint. URL: \url{ http://arxiv.org/abs/1805.06291v2}.

\bibitem[Vas]{bfo-accessible-v1}
Sebastien Vasey, \emph{Accessible categories, set theory, and model theory: an
  invitation}, Preprint. URL: \url{https://arxiv.org/abs/1904.11307v1}.

\bibitem[Vas16]{sv-infinitary-stability-afml}
\bysame, \emph{Infinitary stability theory}, Archive for Mathematical Logic
  \textbf{55} (2016), 567--592.

\end{thebibliography}

\end{document}